\documentclass[a4paper,oneside,11pt]{amsart}
\usepackage{listings}
\usepackage{amsmath}
\usepackage{verbatim}
\usepackage{csquotes}
\usepackage{mathrsfs}
\usepackage[bookmarks=false]{hyperref}
\usepackage{nameref}
\usepackage{color}
\usepackage{xpatch}
\usepackage{chngcntr}
\usepackage{apptools}
\usepackage{accents}
\usepackage{enumitem}
\usepackage{cleveref}
\renewcommand\labelenumi{(\roman{enumi})}
\renewcommand\theenumi\labelenumi

\newtheorem{lem}{Lemma}
\newtheorem{prop}{Proposition}
\newtheorem{thm}{Theorem}

\newtheorem{cor}{Corollary}
\theoremstyle{definition}
\newtheorem{defn}{Definition}
\newtheorem{rem}{Remark}

\newcounter{numl}
\newcommand{\labelnuml}{\textup{(\roman{numl})}}
\newenvironment{numlist}{\begin{list}{\labelnuml}%
{\usecounter{numl}\setlength{\leftmargin}{0pt}%
\setlength{\itemindent}{2\parindent}%
\setlength{\itemsep}{\smallskipamount}\def
\makelabel ##1{\hss \llap {\upshape ##1}}}}{\end{list}}

\DeclareSymbolFont{script}{U}{eus}{m}{n}
\DeclareSymbolFontAlphabet{\mathscr}{script}
\DeclareMathSymbol{\Wedge}{0}{script}{"5E}
\DeclareMathAlphabet{\mathrmsl}{OT1}{cmr}{m}{sl}
\newcommand{\R}{{\mathbb R}}
\newcommand{\C}{{\mathbb C}}

\newcommand{\T}{{\mathbb T}}

\newcommand{\grad}{\mathrm{grad}}

\newcommand{\tor}{{\mathfrak t}}

\newcommand{\Hess}{\mathop{\mathrm{Hess}}}

\newcommand{\bdelta}{\boldsymbol \delta}

\newcommand{\G}{\mathrm G}

\newcommand{\xxi}{\boldsymbol \xi}

\newcommand{\vv}{\mathrm w}
\newcommand{\uu}{\mathrm v}
\newcommand{\ww}{\mathrm u}
\newcommand{\PP}{\mathrm P}

\newcommand{\Autred}{{\rm Aut}_{\rm red}(X)}
\newcommand{\Tr}{{\rm Tr}}
\newcommand{\vol}{\mathrm{vol}}

\newcommand{\Ric}{{\rm Ric}}

\newcommand{\Scal}{{\rm Scal}}

\newcommand{\Ent}{{\rm Ent}}

\newcommand{\PSH}{{\rm PSH}}
\newcommand{\Mab}{\mathcal{M}}

\usepackage[left=3cm,top=3cm,right=3cm,bottom=3cm,bindingoffset=0.5cm]{geometry}

\begin{document}

\author[A. Lahdili]{Abdellah Lahdili}
\address{Beijing International Center for Mathematical Research, Peking University, Beijing, China}
\email{lahdili.abdellah@gmail.com}

\title[Uniqueness of weighted extremal metrics]{Convexity of the weighted Mabuchi functional and the uniqueness of weighted extremal metrics}

\begin{abstract} 
We prove the uniqueness, up to a pull-back by an element of a suitable subgroup of complex automorphisms,  of the weighted extremal K\"ahler metrics on a compact K\"ahler manifold introduced  in our previous work \cite{lahdili3}. This extends a result by Berman--Berndtsson~\cite{BB} and Chen--Paun--Zeng~\cite{CPZ} in the extremal K\"ahler case. Furthermore, we show that a  weighted extremal K\"ahler metric is a global minimum of a suitable weighted version of the  modified Mabuchi energy, thus extending our results from \cite{lahdili3} from the polarized to the K\"ahler case. This implies a suitable notion of weighted K-semistability of a K\"ahler manifold admitting a weighted extremal K\"ahler metric.
\end{abstract}

\maketitle

\section{Introduction}\label{sec-1}

In a previous work \cite{lahdili3}, we introduced  the notion of  a {\it weighted extremal K\"ahler metric} on a K\"ahler manifold $X$,  endowed with a K\"ahler class $\alpha\in H^{1,1}(X,\R)$,  a fixed compact real torus $\T$ inside the reduced group $\Autred$ of complex automorphisms  of $X$,  and two arbitrary smooth positive functions (called weights) $\uu,\vv$ defined on the fixed momentum image $\PP_{\alpha}\subset {\rm Lie}(\T)^*$  for the action of $\T$ with respect to any K\"ahler representative of $\alpha$.  More precisely, given these data,  for any  $\T$-invariant K\"ahler metric $\omega \in \alpha$  with normalized $\T$-momentum map $m_{\omega}: X \to \PP_{\alpha}$,  we define the $\uu$-scalar curvature by
\begin{equation}\label{Scal-v}
\Scal_{\uu}(\omega):=\uu(m_\omega)\Scal(\omega)+2\Delta_\omega(\uu(m_{\omega}))+\Tr(\G_\omega\circ \left(\Hess(\uu)\circ m_{\omega}\right)),
\end{equation}
where $\Scal(\omega)$ is the scalar curvature of $\omega$, $m_\omega:X\to\tor^*$ is the momentum map of the $\T$-action normalized by $m_\omega(X)=\PP_\alpha$, $\Delta_\omega$ is the Riemannian Laplacian of the K\"ahler metric $\omega$ and $\Hess(\uu)$ is the hessian of $\uu$, viewed as bilinear form on $\mathfrak{t}^{*}$ whereas $\G_\omega$ is the bilinear form with smooth coefficients on $\mathfrak{t}$, given by the restriction of the K\"ahler metric $\omega$ on fundamental vector fields.  In a basis $\xxi=(\xi_i)_{i=1,\cdots,\ell}$ of $\tor$, we have
\begin{equation*}
\Tr(\G_\omega\circ \left(\Hess(\uu)\circ m_{\omega}\right))=\sum_{1\leq i,j\leq\ell}\uu_{,ij}(m_\omega)(\xi_i,\xi_j)_\omega,
\end{equation*}
where $\uu_{,ij}$ stands for the partial derivatives of $\uu$ with respect the dual basis of $\xxi$. 

Let $\vv\in C^{\infty}(\PP_\alpha,\R)$ be another smooth positive function on $\PP_\alpha$. Similarly to the approach pioneered by Calabi~\cite{Calabi}, we  are interested to the problem of finding a  $\T$-invariant K\"ahler representative $\omega$  of $\alpha$ for which ${\rm Scal}_{\uu}(\omega)/\vv(m_\omega)$ is the momentum potential of a holomorphic vector field inside the Lie algebra $\tor$ of $\T$. We have shown in \cite{lahdili3} that the problem reduces to solve
\begin{equation}\label{main}
\frac{\Scal_{\uu}(\omega)}{ \vv(m_\omega)}=\ell_{\rm ext}(m_\omega),
\end{equation}
where $\ell_{\rm ext}$ is the $(\uu,\vv)$-extremal affine-linear function on $\tor^*$,  determined from  the data $(\alpha, \T, \PP_{\alpha}, \uu,\vv)$, and we shall refer to a K\"ahler metric satisfying the above condition as a $(\uu, \vv)$-{\it extremal K\"ahler metric} on $(X, \alpha, \T, \PP_{\alpha}, \uu, \vv)$.

Notice that if we take $\T=\{1\}$ and $\uu=\vv \equiv 1$, we obtain the much studied problem of the existence of cscK metric in $\alpha$ whereas taking $\T$ to be a {\it maximal} torus in ${\rm Aut}_{\rm red}(X)$ and $\uu=\vv\equiv 1$,  our problem reduces to the famous Calabi problem of the existence of an extremal K\"ahler metric on $(X, \alpha)$. As we have noticed in \cite{lahdili3}, there is a number of other natural problems in K\"ahler geometry which can be reduced to the search  of  $(\uu,\vv)$-extremal K\"ahler metrics for special choices of $\T$ and the weight functions $\uu$ and $\vv$, including the existence of conformally K\"ahler, Einstein--Maxwell metrics~\cite{AM}, the existence of extremal Sasaki metrics~\cite{AC}, the existence of K\"ahler--Ricci solitons~\cite{Eiji1, Eiji2, B-N}, prescribing the scalar curvature on compact toric manifolds~\cite{Do-02} and on semi-simple, rigid toric fibre bundles~\cite{ACGT} as well as the recently introduced $\mu$-cscK metrics in \cite{Eiji2}.

For a fixed $\T$-invariant K\"ahler metric $\omega\in\alpha$ let $\mathcal{K}(X,\omega)^{\T}$ denote the space of smooth $\T$-invariant K\"ahler potentials with respect to $\omega$, i.e.
\begin{equation*}
\mathcal{K}(X,\omega)^{\T}=\{\phi\in C^\infty(X,\mathbb{R})^{\T}|\omega_\phi=\omega+dd^c\phi>0\}.
\end{equation*}
For $\phi\in\mathcal{K}(X,\omega)^{\T}$ we denote by $m_{\phi} : X \to \tor^*$ the corresponding $\omega_{\phi}$-momentum map, normalized by the condition $m_{\phi}(X)=m_{\omega}(X)=: P_{\alpha}$ or equivalently by $m_{\phi} = m_{\omega} + d^c\phi$ and by $\Scal_\uu(\phi)$ the weighted scalar curvature of $\omega_{\phi}$ introduced by \eqref{Scal-v}. Also, we use the usual convention to denote by $\omega_{\phi}^{[n]}:=\frac{\omega^{n}_\phi}{n!}$ the associated volume form. Following \cite{lahdili3}, for two weight functions $\uu,\vv\in C^{\infty}(\PP_\alpha,\R)$ such that $\uu>0$ and $\vv$ is arbitrary, a K\"ahler potential $\phi\in\mathcal{K}(X,\omega)^{\T}$ defines a $(\uu,\vv)$-weighted cscK metric $\omega_\phi$ if it satisfies
\begin{equation}\label{Eq-(v,w)-cscK}
\Scal_{\uu}(\phi)=c_{\uu,\vv}(\alpha) \vv(m_\phi)
\end{equation}
where $c_{\uu,\vv}(\alpha)$ is a constant depending only on $(\uu,\vv,\alpha)$, given by
\begin{equation}\label{Top-const}
c_{\uu,\vv}(\alpha):=\begin{cases} \frac{\int_X \Scal_{\uu}(\omega)\omega^{[n]}}{\int_X \vv(m_\omega)\omega^{[n]}},&\text{if }\int_X \vv(m_\omega)\omega^{[n]}\neq 0\\ 
1,&\text{if } \int_X \vv(m_\omega)\omega^{[n]}= 0,
\end{cases}
\end{equation} 
The $(\uu,\vv)$-weighted cscK metrics are critical points of the $(\uu,\vv)$-Mabuchi energy $\Mab_{\uu,\vv}:\mathcal{K}(X,\omega)^{\T}\rightarrow\R $ defined on the Fr\'echet space $\mathcal{K}(X,\omega)^{\T}$ by its first variation
\begin{equation}\label{Mabuchi}
\begin{cases} (d\Mab_{\uu,\vv})_{\phi}(\dot{\phi})={\displaystyle-\int_X\dot{\phi}\big(\Scal_{\uu}(\phi)-c_{\uu,\vv}(\alpha)\vv(m_{\phi})\big)\omega_{\phi}^{[n]}},\\ 
\Mab_{\uu,\vv}(0)=0,
\end{cases}
\end{equation}
for all $\dot{\phi}\in T_\phi \mathcal{K}(X,\omega)^{\T}\cong C^{\infty}(X,\R)^{\T}$, where $C^{\infty}(X,\R)^{\T}$ stands for the space of $\T$-invariant smooth functions. As observed in \cite[Section 3.2]{lahdili3}, when $\uu,\vv$ are both positive, a  K\"ahler metric $\omega_{\phi}$ is $(\uu,\vv)$-extremal if and only if it is $(\uu,\ell_{\rm ext}\vv)$-cscK and the relative $(\uu,\vv)$-Mabuchi energy is defined in this case by
\begin{equation}\label{Rel-Mabuchi}
\mathcal{M}_{\uu,\vv}^{\rm rel}=\mathcal{M}_{\uu,\ell_{\rm ext}\vv}.
\end{equation}
where $\ell_{\rm ext}$ is the $(\uu,\vv)$-extremal affine linear function introduced in \cite{lahdili3} via the orthogonal projection of $\Scal_\uu(\phi)$ to the space of (pull-backs by $m_{\phi}$) affine-linear functions on $\tor^*$ with respect to the weighted $L^2$-global product $\langle\varphi_1, \varphi_2\rangle_{\vv,\phi}:= \int_X\varphi_1 \varphi_2\vv(m_\phi)\omega^{[n]}_\phi$. The critical points of the relative $(\uu,\vv)$-Mabuchi energy are precisely the $(\uu,\vv)$-extremal K\"ahler metrics in $\alpha$. 

The space $\mathcal{K}(X,\omega)^{\T}$ is an infinite dimentional Riemannian manifold with a natural Riemanniann metric, called the {\it Mabuchi metric} \cite{Mabuchi1}, defined by 
\begin{equation*}
\langle\dot{\phi}_1,\dot{\phi}_2\rangle_\phi:=\int_X\dot{\phi}_1\dot{\phi}_2\omega^{[n]}_\phi,
\end{equation*}
for any $\phi\in\mathcal{K}(X,\omega)^{\T}$ and $\dot{\phi}_1,\dot{\phi}_2\in C^{\infty}(X,\mathbb{R})^{\T}$. The equation of a geodesic $(\phi_t)_{t\in[0,1]}\in\mathcal{K}(X,\omega)^{\T}$ connecting two points $\phi_0,\phi_1\in\mathcal{K}(X,\omega)^{\T}$ is given by \cite{Mabuchi1, Mabuchi2}
\begin{equation}\label{eq-geod}
\ddot{\phi}_t=|d\dot{\phi}_t|^{2}_{\phi_t}.
\end{equation}
It was shown by Donaldson \cite{Donaldson-geod} and Semmes \cite{Semmes} that by letting $\tau:=e^{-t+is}$, the geodesic $(\phi_t)_{t\in[0,1]}\in\mathcal{K}(X,\omega)^{\T}$ can be viewed as a smooth function $\Phi(x,\tau)$ on $\hat{X}:=X\times\mathbb{A}$, where $\mathbb{A}:=\{e^{-1}\leq|\tau|\leq 1\}$ is the corresponding annulus in $\mathbb{C}$, defined by
\begin{equation}\label{complex}
\Phi(x,\tau):=\phi_{t}(x),
\end{equation}
which is invariant under the natural action of $\mathbb{G}:=\T\times\mathbb{S}^1$ on $\hat{X}$, and satisfies the following degenerate Monge-Amp\`ere equation on $\hat{X}$,
\begin{equation*}
\big(\pi_{X}^*\omega+dd^{c}\Phi\big)^{n+1}=0
\end{equation*}
where $\pi_{X}:\hat{X}\to X$ is the projection on the first factor. Hence, the problem of connecting two potentials $\phi_0,\phi_1\in\mathcal{K}(X,\omega)^{\T}$ by a geodesic $(\phi_t)_{t\in[0,1]}\in\mathcal{K}(X,\omega)^{\T}$ is equivalent to finding a solution $\Phi\in C^\infty(\hat{X},\mathbb{R})^{\mathbb{G}}$ to the following boundary value problem
\begin{equation}\label{MA}
\begin{cases}
\big(\pi_{X}^*\omega+dd^{c}\Phi\big)^{n+1}=0,\\
\omega+dd^{c}\Phi_{|X_\tau}>0,\text{ for }\tau\in\mathbb{A},\\
\Phi(\cdot,e^{-1})=\phi_1\text{ and }\Phi(\cdot,1)=\phi_0.
\end{cases}
\end{equation}
where $X_\tau:=\pi_{\mathbb{A}}^{-1}(\tau)$ is a fiber of the projection $\pi_{\mathbb{A}}:\hat{X}\to \mathbb{A}$. 

In general, the space $\mathcal{K}(X,\omega)^{\T}$ is not geodesically convex by smooth geodesics (see \cite[Theorem 1.2]{Darvas}). However, the boundary value problem \eqref{MA} makes sense for $\mathbb{G}$-invariant bounded plurisubharmonic functions $\Phi\in \PSH(\hat{X},\pi_{X}^{\star}\omega)^{\mathbb{G}}\cap L^{\infty}$, using the Bedford--Taylor interpretation of $\big(\pi_{X}^*\omega+dd^{c}\Phi\big)^{n+1}$ as a Borel measure on $\hat{X}$. 

By a result of Chen \cite{Chen0}, with complements of Blocki \cite{Blocki} and Chu--Tossati--Weinkove \cite{CTW}, the boundary value problem \eqref{MA} admits a unique $\mathbb{G}$-invariant solution $\Phi\in C^{1,1}(\hat{X},\R)$ such that $\pi_{X}^*\omega+dd^{c}\Phi$ is a positive current with bounded coefficients, up to the boundary, corresponding to a family of functions $(\phi_t)_{t\in[0,1]}$ in the space $\mathcal{K}^{1,1}(X,\omega)^{\T}$ of all $\T$-invariant functions $\phi\in C^{1,1}(X,\R)$ such that $\omega_\phi$ is a positive current with bounded coefficients. The curve $(\phi_t)_{t\in[0,1]}\subset\mathcal{K}^{1,1}(X,\omega)^{\T}$ is called the {\it weak geodesic segment} joining $\phi_0,\phi_1\in\mathcal{K}(X,\omega)^{\T}$. Consequently, the space $\mathcal{K}(X,\omega)^{\T}$ is geodesically convex by geodesics in the space $\mathcal{K}^{1,1}(X,\omega)^{\T}$.

Building on the approach by finite dimensional approximations \cite{donaldson1, Li, ST} in the extremal K\"ahler case, we proved in \cite[Corollary~1]{lahdili3} that when $\alpha$ is a polarization, $(\uu,\vv)$-extremal K\"ahler metrics are global minima of $\mathcal M_{\uu,\vv}^{\rm rel}$. In this paper, we extend this result by removing the integrality condition on the K\"ahler class $\alpha$.

To this end, we now follow the approach of Berman--Berndtsson \cite{BB} (see also Chen--Li--Paun \cite{CLP}) who proved that $\mathcal{M}_{1, 1}$ naturally extends to the space $\mathcal{K}^{1,1}(X, \omega)^{\T}$ and is convex along the weak geodesics.  Our main result of this paper is the following.

\begin{thm}\label{thm:main1}
Let $X$ be a compact K\"ahler manifold with K\"ahler class $\alpha$, $\T\subset\Autred$ a real torus with momentum polytope $\PP_\alpha\subset \tor^*$ and $\uu\in C^\infty(\PP_\alpha,\mathbb{R}_{>0})$, $\vv\in C^\infty(\PP_\alpha,\mathbb{R})$. For any $\T$-invariant K\"ahler metric $\omega\in\alpha$, the $(\uu,\vv)$-Mabuchi energy $\mathcal{M}_{\uu,\vv}$ admits a natural extension as functional on the space $\mathcal{K}^{1,1}(X,\omega)^{\T}$ which is convex in the pointwise sense along weak geodesics in $\mathcal{K}^{1,1}(X,\omega)^{\T}$ connecting smooth $\T$-invariant $\omega$-K\"ahler potentials $\phi_{0},\phi_1\in \mathcal{K}(X,\omega)^{\T}$.
\end{thm}

Similarly to the case of cscK metrics, using the sub-slope property of convex functions,  we obtain the following corollary giving an obstruction to the existence of $(\uu,\vv)$-cscK metrics in a K\"ahler class $\alpha$, in terms of the boundedness of the corresponding $(\uu,\vv)$-Mabuchi energy.

\begin{cor}\label{thm-Mab-bound}
Let $\phi_0,\phi_1\in\mathcal{K}(X,\omega)^{\T}$. We have the following inequality 
\begin{equation*}
\Mab_{\uu,\vv}(\phi_1)-\Mab_{\uu,\vv}(\phi_0)\geq -\frac{d(\phi_1,\phi_0)}{\int_{X}\uu(m_\omega)\omega^{[n]}}\parallel\Scal_{\uu}(\phi_0)-\vv(m_{\phi_0})\parallel_{L^{2}(X,\mu_{\phi_0})}
\end{equation*}
where $d$ is the distance corresponding to the Mabuchi metric  and $\parallel\cdot\parallel_{L^2(X,\mu_{\phi_0})}$ is the usual $L^2$-norm on $(X,\mu_{\phi_0})$ with $\mu_{\phi_0}:=\frac{\omega^{[n]}_{\phi_0}}{\vol(X,\alpha)}$. In particular, $(\uu,\vv)$-cscK metrics in a K\"ahler class $\alpha$ minimizes the corresponding $(\uu,\vv)$-Mabuchi energy $\mathcal{M}_{\uu,\vv}$, and any $(\uu, \vv)$-extremal K\"ahler metric in $\alpha$ minimizes the relative weighted Mabuchi energy $\mathcal{M}_{\uu,\vv}^{\rm rel}$. 
\end{cor}
By \cite[Theorem~2]{lahdili3}, we obtain that the weighted K-semistability is a necessary condition for the existence of a $(\uu, \vv)$-extremal K\"ahler metric.
\begin{cor}\label{c:main1} Let $X$ be as in Theorem~\ref{thm:main1}.  If $X$ admits a $\T$-invariant $(\uu,\vv)$-cscK metric in the K\"ahler class $\alpha$, then for any smooth $\T$-equivariant K\"ahler test configuration $(\mathcal X, \mathcal{A})$ of $(X, \alpha)$, which has reduced central fibre, the weighted Futaki invariant ${\mathcal F}_{\uu, \vv} (\mathcal X, \mathcal A)$ defined in \cite{lahdili3} is non-negative.
\end{cor}

Our approach to prove \Cref{thm:main1} closely follows the scheme of Berman--Berndtsson's proof~\cite{BB} in the cscK case (i.e. when $\uu=\vv\equiv 1$). A key point is proving the existence of a natural extension of the $(\uu,\vv)$-Mabuchi energy $\mathcal{M}_{\uu,\vv}$ as a continuous convex functional defined on the space $\mathcal{K}^{1,1}(X,\omega)^{\T}$. To see that a similar extension of $\mathcal{M}_{\uu,\vv}$ exists for arbitrary weights $\uu,\vv$ we use the weighted Chen-Tian decomposition of $\mathcal{M}_{\uu,\vv}$ found in \cite[Theorem 5]{lahdili3},
\begin{equation}\label{Chen-Tian}
\mathcal{M}_{\uu,\vv}(\phi)=\Ent_{\mu_\omega}(\mu_{\uu}(\phi))+\mathcal{E}_{\uu,\vv}(\phi),
\end{equation}
where the first term is an entropy term of the probability measure 
$$\mu_{\uu}(\phi):=\frac{\uu(m_\phi)\omega^{[n]}_\phi}{{\rm vol}(X,\uu(m_\omega)\omega^{[n]})}$$ 
relatively to the reference smooth measure $\mu_\omega:=\frac{\omega^{[n]}}{{\rm vol}(X,\alpha)}$. The second term $\mathcal{E}_{\uu,\vv}$ is an energy type expression given by the integral over $X$ of terms of the form $\phi\ww(\phi)\omega^{j}_\phi\wedge \theta^{n-j}$ where $\theta$ are smooth two forms depending on $\omega$, and $\ww(\phi)$ is a continuous function on $X$ depending on $\uu,\vv$ and $\phi$. The presence of weights introduces an additional difficulty related to the definition and convexity of the momentum map with respect to weak geodesics in the space $\mathcal{K}^{1,1}(X,\omega)^{\T}$. We solve in \Cref{lem-polytope} below by using an approximation argument of Demailly~\cite{Demaily}.
For a weak geodesic $(\phi_t)_{t\in[0,1]}\subset \mathcal{K}^{1,1}(X,\omega)^{\T}$, with $\Phi$ being the corresponding solution of the boundary value problem \eqref{MA} and $\phi_\tau:=\Phi(\cdot,\tau)$ for $\tau\in\mathbb{A}=\{e^{-1}\leq |\tau|\leq 1\}$, Berman--Berndtsson showed in \cite{BB} that the function $\tau\mapsto \mathcal{M}_{1,1}(\phi_\tau)$ is weakly subharmonic on $\mathbb{A}$ and 
\begin{equation*}
dd^c\mathcal{M}_{1,1}(\phi_\tau)=\int_X T.
\end{equation*}
where $T$ is a positive Radon measure on $\hat{X}=X\times\mathbb{A}$ and $\int_X$ denotes the fiber-wise integral on $\pi_\mathbb{A}:\hat{X}\to \mathbb{A}$. In the case when $\uu>0$ and $\vv$ is an arbitrary function on the momentum polytope $\PP_\alpha$, weak subharmonicity of $\tau\mapsto \mathcal{M}_{\uu,\vv}(\phi_\tau)$ on $\mathbb{A}$ will follow from a similar expression
\begin{equation*}
dd^c\mathcal{M}_{\uu,\vv}(\phi_\tau)=\int_X \uu(m_{\phi_\tau})T,
\end{equation*}
and the fact that $\uu(m_{\phi_\tau})T$ is a positive Radon-measure. In particular, $\tau\mapsto \mathcal{M}_{\uu,\vv}(\phi_\tau)$ is weakly convex. To get point-wise convexity, we will show that $\tau\mapsto \mathcal{M}_{\uu,\vv}(\phi_\tau)$ is continuous on $\mathbb{A}$. 

\bigskip
An important application of the approach in \cite{BB} is establishing the uniqueness of the cscK and extremal K\"ahler metrics in $\alpha$,  up to the natural action (by pull-backs) of the connected component of the identity ${\rm Aut}_{\rm red}(X)^{\circ}$ of the reduced group of automorphisms  ${\rm Aut}_{\rm red}(X)$. Similarly, we adapt the proof of the uniqueness of extremal K\"ahler metrics obtained by Chen--Paun--Zeng~\cite{CPZ} to our weighted setting and obtain the following result.
\begin{thm}\label{Thm-uniq} Let $X$ be as in Theorem~\ref{thm:main1} and let $G:={\rm Aut}_{\rm red}^{\T}(X)^{\circ}$ denote the connected component of identity of the commutator of $\T$ inside ${\rm Aut}_{\rm red}(X)$. Then, for any two $\T$-invariant  $(\uu,\vv)$-extremal K\"ahler metrics $\omega_1$ and $\omega_2$ in $\alpha$, there exits an element $f\in G$ such that $\omega_1 = f^*(\omega_2)$.
\end{thm}
Notice that if we take $\T=\{1\}$ and $\uu=\vv\equiv 1$ we get the uniqueness of cscK metrics modulo ${\rm Aut}_{\rm red}(X)^{\circ}$ obtained in \cite{BB,CPZ}, whereas if we take $\T$ to be a maximal torus inside ${\rm Aut}_{\rm red}(X)$ and $\uu=\vv \equiv 1$, the above results yield the uniqueness of the $\T$-invariant extremal K\"ahler metrics modulo the complexification $\T^{\C}$ of $\T$.

\section*{Acknowledgement} 
I am grateful to V. Apostolov for his invaluable advice. Thanks to E. Inoue for his interest and discussions, and more specific thanks to G. Tian and X. Zhu for their interest and support. I would like to thank BICMR for the financial support.

\section{The weighted Mabuchi energy on $\mathcal{K}^{1,1}(X,\omega)^{\T}$}

Let $X$ be a compact K\"ahler manifold of complex dimension $n\geq 2$. We denote by $\Autred$ the reduced automorphism group of $X$ whose Lie algebra $\mathfrak{h}_{\rm red}$ is given by real holomorphic vectorfields with zeros (see \cite{Gauduchon}). Let $\T$ be an $\ell$-dimentional real torus in $\Autred$ with Lie algebra $\tor$, and $\omega$ a fixed $\T$-invariant K\"ahler form on $X$. The $\T$-action on $X$ is $\omega$-Hamiltonian (see \cite{Gauduchon}) and we choose $m_\omega : X\to \tor^{*}$ to be a $\omega$-momentum map of $\T$. It is well known \cite{Atiyah, Guil-Stern} that $\PP_\omega= m_\omega(X)$ is a convex polytope in $\tor$. For any smooth $\T$-invariant $\omega$-K\"ahler potential $\phi\in\mathcal{K}(X,\omega)^{\T}$, let $\PP_\phi:=m_\phi(X)$ be the $\omega_\phi$-momentum image of $X$. By \cite{Atiyah, Guil-Stern}, the following two facts are equivalent:
\begin{enumerate}
\item $\PP_\phi=\PP_\omega$.
\item\label{(ii)-mom} $\langle m_\phi,\xi\rangle=\langle m_\omega,\xi\rangle+(d^c\phi)(\xi)$ for any $\xi\in\tor$.
\end{enumerate}
It follows that we can normalize $m_\phi$ such that $\PP_\phi=\PP_\omega$ is a $\phi$-independent polytope $\PP_\alpha\subset\tor^*$. For $\phi\in\mathcal{K}^{1,1}(X,\omega)^{\T}$ the space of all $\T$-invariant functions $\phi\in C^{1,1}(X,\R)$ such that $\omega_\phi$ is a positive current with bounded coefficients, we define $m_\phi:X\to\tor^{*}$ by
\begin{equation*}
\langle m_\phi,\xi\rangle=\langle m_\omega,\xi\rangle+(d^c\phi)(\xi),
\end{equation*} 
for any $\xi\in\tor$. 
\begin{lem}\label{lem-polytope}
For any $\phi\in\mathcal{K}^{1,1}(X,\omega)^{\T}$, we have $\PP_{\phi}=\PP_\alpha$.
\end{lem}
\begin{proof}
For any $k>0$ we have $\omega_k=k\omega+\omega_\phi>0$. By \cite[Theorem 5.21]{Demaily} we can find a decreasing sequence $\phi_\epsilon\in C^{\infty}(X,\mathbb{R})^{\T}$ such that $\omega_{k,\epsilon}:=(k+1)\omega+dd^{c}\phi_{\epsilon}$ is K\"ahler and $\phi_\epsilon\to \phi$ in $C^1$ topology as $\epsilon\to 0$. For any $\epsilon>0$, we have $\PP_{\omega_{k,\epsilon}}=(k+1)\PP_\alpha$, and for any $\xi\in\tor$ we have
\begin{equation}\label{mom-eps}
\langle m_{\omega_{k,\epsilon}},\xi\rangle=(k+1)\langle m_{\omega},\xi\rangle+(d^c\phi_{\epsilon})(\xi).
\end{equation}
Since $\phi_{\epsilon}$ converge to $\phi$ in $C^1$ topology, passing to the limit when $\epsilon\to 0$ in \eqref{mom-eps}, we obtain  
$$\langle m_{\omega_{k,\epsilon}},\xi\rangle\to(k+1)\langle m_\omega,\xi\rangle+(d^c\phi)(\xi)=\langle km_\omega+m_\phi,\xi\rangle,$$ 
as $\epsilon\to 0$, for $\xi\in \tor$ fixed. It follows that 
$$(k+1)\PP_\alpha=\PP_{\omega_{k,\epsilon}}=(\underset{\epsilon\to 0}{\lim}\, m_{\omega_{k,\epsilon}})(X)=(km_\omega+m_\phi)(X)=k\PP_\alpha+\PP_\phi.$$
The result follows by taking the limit when $k\to 0$.
\end{proof}
\begin{rem}
In \cite{Chen0}, Chen considered the following family of elliptic boundary value problems with parameter $\epsilon>0$,
\begin{equation}\label{MA-epsilon}
\begin{cases}
\big(\pi_{X}^*\omega+dd^{c}\Phi^{\epsilon}\big)^{n+1}=\epsilon\big(\pi_{X}^*\omega+\frac{\sqrt{-1}d\tau\wedge d\bar\tau}{2|\tau|^{2}}\big)^{n+1},\\
\Phi^{\epsilon}(\cdot,e^{-1})=\phi_1\text{ and }\Phi^{\epsilon}(\cdot,1)=\phi_0.
\end{cases}
\end{equation}
Solutions $\Phi^\epsilon\in\mathcal{K}(\hat{X},\pi^{*}_X\omega)^{\mathbb{G}}$ of \eqref{MA-epsilon}, are always smooth and approximate uniformly the weak solution $\Phi$ of \eqref{MA}. More precisely, $\Phi^{\epsilon}$ is decreasing in $\epsilon$ and converges to the solution $\Phi$ of \eqref{MA} in the weak $C^{1,1}$ topology as $\epsilon\to 0$ (see \cite[Lemma 7]{Chen0}). The family of K\"ahler potentials $(\phi^{\epsilon}_t)_{t\in[0,1]}\subset\mathcal{K}(X,\omega)^{\T}$ is called an $\epsilon$-{\it geodesic}.

If $(\phi_t)_{t\in[0,1]}\in\mathcal{K}^{1,1}(X,\omega)^{\T}$ is a weak geodesic  segment, one can show that $\PP_{\phi_t}=\PP_\omega$ for any $t\in [0,1]$ using the fact that the $\epsilon$-geodesic $(\phi^{\epsilon}_t)_{t\in[0,1]}\in\mathcal{K}(X,\omega)^{\T}$ converges to $\phi$ in the weak $C^{1,1}$-topology as $\epsilon\to 0$, together with the relation 
$$m_{\phi^{\epsilon}_t}=m_{\omega}+d^c\phi^{\epsilon}_t.$$
\end{rem}
Let $\uu\in C^{\infty}(\PP_\alpha,\mathbb{R}_{>0})$ and $\vv\in C^{\infty}(\PP_\alpha,\mathbb{R})$ two smooth functions. Now, we give the energy functionals allowing to define the $(\uu,\vv)$-Mabuchi energy \eqref{Mabuchi} on weak geodesic segments.  
\begin{lem}\label{E-lem-def}
The functional $\mathcal{E}_{\vv}:\mathcal{K}(X,\omega)^{\T}\to\mathbb{R}$ given by 
\begin{equation}\label{E}
\begin{cases} \left(d\mathcal{E}_{\vv}\right)_{\phi}(\dot{\phi})={\displaystyle\int_X\dot{\phi}\vv(m_{\phi})\omega_{\phi}^{[n]}},\\ 
\mathcal{E}_{\vv}(0)=0,
\end{cases}
\end{equation} 
for any $\dot{\phi}\in T_\phi\mathcal{K}(X,\omega)^{\T}\cong C^{\infty}(X,\mathbb{R})^{\T}$ is well-defined and has a natural  extension to $\mathcal{K}^{1,1}(X,\omega)^{\T}$.
\end{lem}
\begin{proof}
The first claim in the Lemma is well known (see, for example, \cite[Proposition 2.16]{B-N}). Now we will extend $\mathcal{E}_{\vv}:\mathcal{K}(X,\omega)^{\T}\to\mathbb{R}$ to $\mathcal{K}^{1,1}(X,\omega)^{\T}$. We have  
\begin{align*}
\mathcal{E}_{\vv}(\phi)=&\int^{1}_0\left(\int_X\phi\vv(m_{\epsilon\phi})\omega^{[n]}_{\epsilon\phi}\right)d\epsilon\\
=&\int^{1}_0\left(\int_X\sum_{j=0}^{n}\phi\epsilon^{n-j}(1-\epsilon)^j\vv(\epsilon m_\phi+(1-\epsilon)m_\omega)\omega^{[n-j]}_\phi\wedge\omega^{[j]}\right)d\epsilon\\
=&\int_X\phi\sum_{j=0}^{n}\vv_{j,n}(m_\phi)\omega^{[n-j]}_\phi\wedge\omega^{[j]},
\end{align*}
where $\vv_{j,n}:\PP_\alpha\to\mathbb{R}$ is defined by
\begin{equation}\label{j-n-trans}
\vv_{j,n}(p):=\int^{1}_0\epsilon^{n-j}(1-\epsilon)^j\vv(\epsilon p+(1-\epsilon)m_\omega)d\epsilon.
\end{equation}
Using the expression
\begin{equation}\label{ext-E-v}
\mathcal{E}_{\vv}(\phi)=\int_X\phi\sum_{j=0}^{n}\vv_{j,n}(m_\phi)\omega^{[n-j]}_\phi\wedge\omega^{[j]},
\end{equation}
we can define the extension $\mathcal{E}_{\vv}:\mathcal{K}^{1,1}(X,\omega)^{\T}\to\mathbb{R}$, since by \Cref{lem-polytope} we have $\epsilon m_\phi+(1-\epsilon)m_\omega\in \PP_\alpha$ by convexity.
\end{proof}

\begin{lem}\label{ddc-Ev}
Let $(\phi_t)_{t\in[0,1]}\subset\mathcal{K}^{1,1}(X,\omega)^{\T}$ be a geodesic segment connecting $\phi_{0},\phi_{1}\in\mathcal{K}(X,\omega)^{\T}$ and $\Phi\in \mathcal{K}^{1,1}(\hat{X},\pi^{*}_X\omega)^{\mathbb{G}}$ the corresponding solution of the boundary value problem \eqref{MA}. For any $\tau\in\mathbb{A}$, we have
\begin{equation*}
dd^{c}\mathcal{E}_{\vv}(\phi_\tau)=0,
\end{equation*}  
where $\phi_\tau:=\Phi(\cdot,\tau)$. 
\end{lem} 
\begin{proof}
See \cite[Proposition 17]{B-N}.
\end{proof}
\begin{defn}\label{theta-mom}
Let $\theta$ be a $\T$-invariant closed $(1,1)$-form on $X$. A $\theta$-momentum map for
the action of $\T$ on $X$ is a smooth $\T$-invariant function $m_\theta : X\to \tor^*$ with the property $\theta(\xi,\cdot)=-d\langle m_\theta,\xi\rangle$ for all $\xi\in\tor$.
\end{defn}
For example, if $\Ric(\omega)$ is the Ricci form of $\omega$, then the $\Ric(\omega)$-momentum map for the action of $\T$ on $X$ is given by (see e.g. \cite[Lemma 5]{lahdili3})
$$m_{\Ric(\omega)}:=\frac{1}{2}\Delta_\omega(m_\omega).$$  
\begin{lem}\label{Lem-E-theta}\cite[Lemma 4]{lahdili3}
Let $\theta$ be a fixed $\T$-invariant closed $(1,1)$-form and $m_\theta:X\rightarrow \mathfrak{t}^{*}$ a momentum map with respect to $\theta$, see \Cref{theta-mom}. Then the functional $\mathcal{E}_{\uu}^{\theta}:\mathcal{K}(X,\omega)^{\T}\rightarrow\mathbb{R}$ given by 
\begin{equation}\label{E-theta}
\begin{cases} (d\mathcal{E}_{\uu}^{\theta})_{\phi}(\dot{\phi})={\displaystyle\int_X \dot{\phi} \left[\uu(m_{\phi})\theta\wedge\omega^{[n-1]}_\phi+\langle (d\uu)(m_{\phi}),m_{\theta}\rangle\omega^{[n]}_{\phi}\right]},\\ 
\mathcal{E}_{\uu}^{\theta}(0)=0,
\end{cases}
\end{equation}
for any $\dot{\phi}\in  C^{\infty}(X,\mathbb{R})^{\T}$ is well-defined and has a natural extension to $\mathcal{K}^{1,1}(X,\omega)^{\T}$.
\end{lem}
\begin{proof}
Similarly to $\mathcal{E}_{\vv}$, we can define the extension $\mathcal{E}_{\uu}^{\theta}:\mathcal{K}^{1,1}(X,\omega)^{\T}\to \mathbb{R}$, by using the following expression
\begin{align}
\begin{split}\label{ext-E-u-theta}
\mathcal{E}_{\uu}^{\theta}(\phi)=&\int^{1}_0\left(\int_X\phi\left[\uu(m_{\epsilon\phi})\theta\wedge\omega^{[n-1]}_{\epsilon\phi}+\langle (d\uu)(m_{\epsilon\phi}),m_{\theta}\rangle\omega^{[n]}_{\epsilon\phi}\right]\right)d\epsilon\\
=&\int_X\phi\Big[\sum_{j=0}^{n-1}\uu_{j,n-1}(m_\phi)\omega^{[n-1-j]}_\phi\wedge\omega^{[j]}\wedge \theta+\sum_{j=0}^{n}\langle (d\uu_{j,n})(m_{\phi}),m_{\theta}\rangle\omega^{[n-j]}_{\phi}\wedge\omega^{[j]}\Big],
\end{split}
\end{align}
where $\uu_{j,n}:\PP_\alpha\to\mathbb{R}$ is given by \eqref{j-n-trans}. 
\end{proof}
Now we give the Chen--Tian formula allowing to extend the $(\uu,\vv)$-Mabuchi energy to $\mathcal{M}_{\uu,\vv}$ to $\mathcal{K}^{1,1}(X,\omega)^{\T}$.
\begin{thm}\cite[Theorem 5]{lahdili3}\label{Chen-Tian}
We have the following expression for the $(\uu,\vv)$-Mabuchi energy,
\begin{equation}\label{C-T}
\mathcal{M}_{\uu,\vv}=\mathcal{H}_{\uu}-2\mathcal{E}^{\Ric(\omega)}_{\uu}+c_{(\uu,\vv)}(\alpha)\mathcal{E}_{\vv},
\end{equation}
on $\mathcal{K}(X,\omega)^{\T}$ where $\mathcal{H}_{\uu}:\mathcal{K}(X,\omega)^{\T}\to \mathbb{R}$ is given
 \begin{equation}\label{Entropy}
\mathcal{H}_{\uu}(\phi):=\int_X\log\left(\frac{\omega^{n}_\phi}{\omega^{n}}\right)\uu(m_{\phi})\omega^{[n]}_\phi={\rm Ent}_{\mu_\omega}(\mu_\uu(\phi))+c(\alpha,\uu).
\end{equation}
where 
$${\rm Ent}_{\mu_\omega}(\mu_\uu(\phi)):=\int_X \log\left(\frac{d\mu_{\uu}(\phi)}{d\mu_{\omega}}\right)\frac{d\mu_{\uu}(\phi)}{d\mu_{\omega}}d\mu_{\omega}$$ 
is the entropy of the probability measure $\mu_{\uu}(\phi):=\frac{\uu(m_\phi)\omega^{[n]}_\phi}{{\rm vol}(X,\uu(m_\omega)\omega^{[n]})}$ relatively to the reference smooth measure $\mu_\omega:=\frac{\omega^{[n]}}{{\rm vol}(X,\alpha)}$ with , and $c(\alpha,\uu)$ is a constant depending on $\alpha$ and $\uu$.
\end{thm}
For $\phi\in \mathcal{K}^{1,1}(X,\omega)^{\T}$, $\mu_{\uu}(\phi)$ is a measure with bounded coefficient which is absolutely continuous with respect to $\mu_\omega$, thus ${\rm Ent}_{\mu_\omega}(\mu_\uu(\phi))$ is well defined on $\mathcal{K}^{1,1}(X,\omega)^{\T}$. Combining this with Lemmas \ref{E-lem-def} and \ref{Lem-E-theta} yields the following.
\begin{cor}
The equation \eqref{Chen-Tian}, extends the $(\uu,\vv)$-Mabuchi energy $\mathcal{M}_{\uu,\vv}$ to a functional on the space $\mathcal{K}^{1,1}(X,\omega)^{\T}$.
\end{cor}

\section{Convexity of the $(\uu,\vv)$-Mabuchi energy along weak geodesics}
The following formulas allow us to compute the second variations of the energy functionals $\mathcal{E}^{\theta}_{\uu}$ and $\mathcal{E}_{\vv}$ along weak geodesics.
\begin{lem}\label{Lem-Monge}
Let $\Phi$ be a $\mathbb{G}$-invariant smooth function on $\hat{X}$ related to a family of $\T$-invariant functions $(\phi_t)_{t\in[0,1]}$ on $X$ by \eqref{complex}, and $\theta$ a $2$-form on $X$. We have
\begin{align*}
(\pi^{*}_X\omega+dd^{c}\Phi)^{[n+1]}=&-\big(\ddot{\phi}_t-|d\dot{\phi}_t|_{\phi_t}^{2}\big)\omega_{\phi_t}^{[n]}\wedge dt\wedge ds,\\
\pi^{*}_X\theta\wedge(\pi^{*}_X\omega+dd^{c}\Phi)^{[n]}=&-\big(\big(\ddot{\phi}_t-|d\dot{\phi}_t|^{2}_{\phi_t}\big)\theta\wedge\omega^{[n-1]}_{\phi_t}+(\theta,d\dot{\phi_t}\wedge d^{c}\dot{\phi_t})_\phi\,\omega^{[n]}_{\phi_t}\big)\wedge dt\wedge ds.
\end{align*} 
where $\dot{\phi}_t$ and $\ddot{\phi}_t$ are the $t$-derivatives of $\phi_t$.
\end{lem}
\begin{proof}
We have $dd^{c}\Phi=dd^{c}\phi+\gamma_\phi$ such that 
\begin{equation*}
\gamma_\phi:=-d^{c}\dot{\phi}\wedge dt-d\dot{\phi}\wedge ds-\ddot{\phi} dt\wedge ds.
\end{equation*} 
By a straightforward calculation we get $\gamma^{2}_\phi=2d\dot{\phi}\wedge d^{c}\dot{\phi}\wedge dt\wedge ds$ and
$\gamma_{\phi}^3=0$. We calculate
\begin{align*}
(\omega+dd^{c}\Phi)^{[n+1]}=&(\omega_{\phi}+\gamma_\phi)^{[n+1]}\\
=&\omega_{\phi}^{[n+1]}+\omega_{\phi}^{[n]}\wedge\gamma_\phi+\frac{1}{2}\omega_{\phi}^{[m-1]}\wedge\gamma^{2}_\phi\\
=&-\big(\ddot{\phi}-|d\dot{\phi}|_{\phi}^{2}\big)\omega_{\phi}^{[n]}\wedge dt\wedge ds.
\end{align*}
For the second identity
\begin{align}
\begin{split}
\theta\wedge(\omega+dd^{c}\Phi)^{[n]}=&\theta\wedge\big[\omega^{[n]}_\phi+\omega^{[n-1]}_\phi\wedge\gamma_{\phi}+\frac{1}{2}\omega^{[n-2]}_\phi\wedge\gamma_{\phi}^{2}\big]\\
=&\theta\wedge\omega^{[n-1]}_\phi\wedge\gamma_\phi+\frac{1}{2}\theta\wedge\omega^{[n-2]}_\phi\wedge\gamma_{\phi}^{2}\\
=&-\big((\ddot{\phi}-|d\dot{\phi}|^{2})\theta\wedge\omega^{[n-1]}_\phi+(\theta,d\dot{\phi}\wedge d^{c}\dot{\phi})\omega^{[n]}_\phi\big)\wedge dt\wedge ds.
\end{split}
\end{align}
\end{proof}
We start by computing the second variation of $\mathcal{E}_{\uu}^\theta$ and $\mathcal{E}_{\vv}$ on smooth families of smooth $\T$-invariant K\"ahler potentials. 
\begin{lem}\label{ddc-E}
Let $(\phi_t)_{t\in[0,1]}\in\mathcal{K}(X,\omega)^{\T}$ be a smooth family of K\"ahler potentials and $\Phi$ the $\mathbb{G}$-invariant function on $\hat{X}$, corresponding to $(\phi_t)_{t\in[0,1]}$ given by \eqref{complex}. Let $\phi_\tau:=\Phi(\cdot,\tau)$. 
\begin{enumerate}
\item\label{item-E} The second variation of the function $\tau\mapsto\mathcal{E}_{\vv}(\phi_\tau)$ on $\mathbb{A}$ is given by
\begin{align}
\begin{split}\label{eq-ddc-E}
dd^{c}\mathcal{E}_{\vv}(\tau)=&\int_X \vv(m_{\Phi})\left(\pi^{*}_X\omega+dd^{c}\Phi\right)^{[n+1]},
\end{split}
\end{align}
where $m_\Phi(x,\tau):=m_{\phi_\tau}$, and $\int_X$ is the push forward map on $\pi_\mathbb{A}:\hat{X}\to \mathbb{A}$.
\item\label{item-E-theta} The second variation of the function $\tau\mapsto\mathcal{E}_{\uu}^\theta(\phi_\tau)$ on $\mathbb{A}$ is given by
\begin{align}
\begin{split}\label{eq-ddc-E-theta}
dd^{c}\mathcal{E}^{\theta}_{\uu}(\phi_\tau)=\int_X\uu(m_{\Phi})\pi^{*}_X\theta\wedge&(\pi^{*}_X\omega+dd^{c}\Phi)^{[n]}+\langle d\uu(m_{\Phi}),m_{\theta}\rangle(\pi^{*}_X\omega+dd^{c}\Phi)^{[n+1]}.
\end{split}
\end{align}
\end{enumerate}
\end{lem}

\begin{proof}
The proof of \ref{item-E} is given in \cite[Proposition 10.d]{Bern}. For \ref{item-E-theta}, by the $\mathbb{S}^{1}$-invariance of $\Phi$, we have
\begin{equation}\label{dd}
dd^{c}\mathcal{E}^{\theta}_{\uu}(\phi_\tau)=\left[\frac{\partial^{2}}{\partial\tau\partial\bar{\tau}}\mathcal{E}^{\theta}_{\uu}(\tau)\right]d\tau \wedge d\bar{\tau}=-\left[\frac{d^{2}}{dt^{2}}\mathcal{E}^{\theta}_{\uu}(\phi_t)\right]dt\wedge ds
\end{equation}
Using \cite[equation (18)]{lahdili3} and \eqref{E-theta} we get
\begin{align}
\begin{split}\label{eq-d-E-theta}
\frac{d^{2}}{dt^{2}}\mathcal{E}^{\theta}_{\uu}(\phi_t)=&(\bdelta\mathcal{B}_{\uu}(\dot{\phi}))_{\phi_t}(\dot{\phi})+\int_X\ddot{\phi} \uu(m_{\phi})\theta\wedge\omega^{[n-1]}_\phi+\langle d\uu(m_{\phi}),m_{\theta}\rangle\ddot{\phi}\omega^{[n]}_{\phi}\\
=&\int_X(\ddot{\phi}-|d\dot{\phi}|^{2}_\phi) \uu(m_{\phi})\theta\wedge\omega^{[n-1]}_\phi+\int_X(\ddot{\phi}-|d\dot{\phi}|^{2}_\phi)\langle d\uu(m_{\phi}),m_{\theta}\rangle\omega^{[n]}_{\phi}\\&+\int_X(\theta,d\dot\phi\wedge d^{c}\dot{\phi})_\phi\uu(m_{\phi})\omega^{[n]}_\phi.
\end{split}
\end{align}
The identity \eqref{eq-ddc-E-theta} follows from \Cref{Lem-Monge} and the equalities \eqref{dd} and \eqref{eq-d-E-theta}.
\end{proof}

Now we consider the second variations along a weak geodesic segment.

\begin{lem}\label{lem-eq-ddc-func}
Let $(\phi_t)_{t\in[0,1]}\in\mathcal{K}^{1,1}(X,\omega)^{\T}$ be a weak geodesic segment and $\Phi$ the $\mathbb{G}$-invariant corresponding solution of the boundary value problem \eqref{MA} on $\hat{X}$. The following identities holdes in the weak sense of currents
\begin{align}
dd^{c}\mathcal{E}_{\vv}(\phi_\tau)=&0,\label{eq-a}\\
dd^{c}\mathcal{E}^{\theta}_{\uu}(\phi_\tau)=&\int_X\uu(m_{\Phi})\pi^{*}_X\theta\wedge(\pi^{*}_X\omega+dd^{c}\Phi\big)^{[n]}.\label{eq-b}
\end{align}
\end{lem}

\begin{proof}
The equation \eqref{eq-a} is already established in \cite[Proposition 2.16]{BB} and \cite[Proposition 10.4]{Bern}. The same argument works for \eqref{eq-b}. We give the proof for convenience of the reader. Let $(\phi^{\epsilon}_t)_{t\in [0,1]}$ be the $\epsilon$-geodesic approximating $(\phi_t)_{t\in [0,1]}$ and $\Phi^{\epsilon}$ the corresponding solution of the elliptic Direchlet problem \eqref{MA-epsilon}. By \Cref{ddc-E}, we have
\begin{align}
\begin{split}\label{eq-ddd}
dd^{c}\mathcal{E}_{\vv}(\phi^{\epsilon}_\tau)=&\int_X \epsilon\vv(m_{\Phi^{\epsilon}})\Big(\pi^{*}_X\omega+\frac{\sqrt{-1}d\tau\wedge d\bar\tau}{2|\tau|^{2}}\Big)^{[n+1]},\\
dd^{c}\mathcal{E}^{\theta}_{\uu}(\phi^{\epsilon}_\tau)=&\int_X\uu(m_{\Phi^\epsilon})\pi^{*}_X\theta\wedge(\pi^{*}_X\omega+dd^{c}\Phi^\epsilon)^{[n]}+\epsilon\langle d\uu(m_{\Phi^\epsilon}),m_{\theta}\rangle\Big(\omega+\frac{\sqrt{-1}d\tau\wedge d\bar\tau}{2|\tau|^{2}}\Big)^{[n+1]}.
\end{split}
\end{align}
We have $\Phi^{\epsilon}$ is decreasing in $\epsilon>0$ and $\Phi^{\epsilon}\rightarrow \Phi$ in $(C^{1,1},\parallel\cdot\parallel_{C^1}+\parallel dd^{c}\cdot\parallel_{L^{\infty}})$ when $\epsilon\rightarrow 0$. Using the identity
\begin{equation*}
m_{\phi^{\epsilon}_\tau}=m_{\omega}+d^{c}\phi^{\epsilon}_\tau,
\end{equation*} 
and the fact that $\uu$ is smooth on $\PP_\alpha$, we obtain
$$\mathcal{E}_{\vv}(\phi^{\epsilon}_\tau)\rightarrow \mathcal{E}_{\vv}(\phi_\tau)\text{ and }\mathcal{E}_{\uu}^{\theta}(\phi^{\epsilon}_\tau)\rightarrow \mathcal{E}^{\theta}_{\uu}(\phi_\tau),$$
since the Monge-Amp{\`e}re measures converges weakly under decreasing limits. It follows that 
$$dd^c\mathcal{E}_{\vv}(\phi^{\epsilon}_\tau)\rightarrow dd^c\mathcal{E}_{\vv}(\phi_\tau)\text{ and }dd^c\mathcal{E}_{\uu}^{\theta}(\phi^{\epsilon}_\tau)\rightarrow dd^c\mathcal{E}^{\theta}_{\uu}(\phi_\tau),$$
in the weak sense of distributions. Passing to the limit when $\epsilon\to 0$ in the rhs of the equations of \eqref{eq-ddd}, and using the fact that $\pi^{*}_X\theta\wedge(\pi^{*}_X\omega+dd^{c}\Phi^\epsilon)^{[n]}\to \pi^{*}_X\theta\wedge(\pi^{*}_X\omega+dd^{c}\Phi)^{[n]}$ in the sense of measures (since $\Phi^\epsilon\searrow\Phi$), we obtain \eqref{eq-a} and  \eqref{eq-b}.
\end{proof}

\begin{cor}\label{Cor-E-theta-strict}
Let $\theta$ be a $\T$-invariant K\"ahler form. The functional $\mathcal{E}_{\uu}^{\theta}$ is strictly convex on weak geodesic segments. In particular $\mathcal{E}_{\uu}^{\theta}$ has at most one critical point in $\mathcal{K}^{1,1}(X,\omega)^{\T}$.
\end{cor}
\begin{proof}
Using \eqref{eq-d-E-theta}, we see that the following formula holds on any weak geodesic segment
\begin{equation*}
\frac{d^2}{dt^2}\mathcal{E}_{\uu}^{\theta}(\phi_t)=\int_X(\theta, d\dot{\phi}\wedge d^{c}\dot{\phi})_{\phi_t}\uu(m_{\phi_t})\omega^{[n]}_{\phi_t}>0,
\end{equation*} 
since $\theta$ is a K\"ahler form. Thus, $t\mapsto\mathcal{E}_{\uu}^{\theta}(\phi_t)$ is strictly convex. 
\end{proof}

Now we consider the entropy part of the $(\uu,\vv)$-Mabuchi energy. For a family of $\mathbb{G}$-invariant volume forms $\Theta_\tau$ on $X$ we associate a function $\Psi:=\log(\Theta_\tau)$ on $\hat{X}$, given locally on a holomorphic coordinate patch $(U,(z_j)_{j=1,n})$ on $X$ by
\begin{equation}\label{Psi}
\Psi_{U}=\log\left(\frac{\Theta_\tau}{{\rm vol_U}}\right),
\end{equation}
where ${\rm vol_U}$ is the volume form of the flat K\"ahler metric $\frac{\sqrt{-1}}{2}\sum_{j=1}^{n}dz_j\wedge d\bar{z}_j$ on $U$. For $(\phi_\tau)_{\tau\in\mathbb{A}}\subset\mathcal{K}^{1,1}(X,\omega)^{\T}$, we define 
\begin{align}\label{H-Psi}
\begin{split}
\mathcal{H}^{\Psi}_{\uu}(\phi_\tau):=&\int_X\log\big(\frac{\Theta_\tau}{\omega^{n}}\big)\uu(m_{\phi_\tau})\omega_{\phi_\tau}^{[n]}\\
=&\int_X\log\big(\frac{e^{\psi_\tau}}{\omega^{n}}\big)\uu(m_{\phi_\tau})\omega_{\phi_\tau}^{[n]},
\end{split}
\end{align}
where $\psi_\tau:=\Psi_{|X_\tau}$.

\begin{lem}\label{ddc-H}
Let $(\phi_t)_{t\in[0,1]}\in\mathcal{K}^{1,1}(X,\omega)^{\T}$ be a weak geodesic segment and denote by $\Phi$ the associated $\mathbb{G}$-invariant function on $\hat{X}$. If $\Psi$ (given by \eqref{Psi}) is smooth, then we have
\begin{equation}\label{eq-ddc-H}
dd^{c}\big(\mathcal{H}^{\Psi}_{\uu}(\phi_\tau)-2\mathcal{E}^{\Ric(\omega)}_{\uu}(\phi_{\tau})\big)=\int_X\uu(m_{\Phi})dd^{c}\Psi\wedge\big(\pi^{*}_X\omega+dd^{c}\Phi\big)^{[n]},
\end{equation}
in the weak sense of currents.
\end{lem}

\begin{proof}
Let $f(\tau)$ be a test function with support in $\mathbb{A}$ and $\hat{f}:=\pi^{\star}_{\mathbb{A}}f$. We have
\begin{align*}
\langle dd^{c}\mathcal{H}^{\Psi}(\phi_\tau),f\rangle=&\int_{\mathbb{A}}dd^{c}f\int_X\log\big(\frac{e^{\psi_\tau}}{\omega^{n}}\big)\uu(m_{\phi_\tau})\omega_{\phi_\tau}^{[n]}\\
=&\int_{\mathbb{A}}dd^{c}f\int_{X_\tau} \Big(\log\big(\frac{e^{\Psi}}{\pi^{*}_X\omega^{n}}\big)\uu(m_{\Phi})(\pi^{*}_X\omega+dd^{c}\Phi)^{[n]}\Big)_{|X_\tau}\\
=&\int_{\hat{X}} \log\big(\frac{e^{\Psi}}{\pi^{*}_X\omega^{n}}\big)\uu(m_{\Phi})dd^{c}\hat{f}\wedge(\pi^{*}_X\omega+dd^{c}\Phi)^{[n]}\\
=&\int_{\hat{X}}\log\big(\frac{e^{\Psi}}{\pi^{*}_X\omega^{n}}\big) d(\uu(m_{\Phi}))\wedge d^{c}\hat{f}\wedge(\pi^{*}_X\omega+dd^{c}\Phi)^{[n]}\\
&-\int_{\hat{X}} \uu(m_{\Phi})d\log\big(\frac{e^{\Psi}}{\pi^{*}_X\omega^{n}}\big)\wedge d^{c}\hat{f}\wedge(\pi^{*}_X\omega+dd^{c}\Phi)^{[n]}
\end{align*}
Notice that $d(\uu(m_{\Phi}))\wedge d^{c}\hat{f}\wedge(\pi^{*}_X\omega+dd^{c}\Phi)^{[n]}=0$, using approximation by an $\epsilon$-geodesic $\Phi^{\epsilon}$ and the fact that $d\hat{f}$ is zero on fundamental vector fields of the $\T$-action: Indeed
\begin{align*}
d(\uu(m_{\Phi^\epsilon}))\wedge d^{c}\hat{f}\wedge(\pi^{*}_X\omega+dd^{c}\Phi^\epsilon)^{[n]}=&\langle (d\uu)(m_{\Phi^\epsilon}),(d\hat{f})_{\tor}\rangle(\pi^{*}_X\omega+dd^{c}\Phi^\epsilon)^{[n+1]}=0,
\end{align*}
since $(d\hat{f})_{\tor}$ the restriction of $d\hat{f}$ on the fundamental vector fields of $\tor$ is zero ($(d\hat{f})_{\tor}=(df\circ\pi_*)_{\tor}=0$). Using that $\Phi^\epsilon\searrow\Phi$ in $C^{1,1}$ topology, passing to the limit as $\epsilon\to 0$, yields $d(\uu(m_{\Phi}))\wedge d^{c}\hat{f}\wedge(\pi^{*}_X\omega+dd^{c}\Phi)^{[n]}=0$.\\
Integration by parts gives
\begin{align*}
\langle dd^{c}\mathcal{H}^{\Psi}_{\uu},f\rangle=&\int_{\hat{X}}\hat{f}d\log\big(\frac{e^{\Psi}}{\pi^{*}_X\omega^{n}}\big)\wedge d^c \uu(m_{\Phi})\wedge (\pi^{*}_X\omega+dd^{c}\Phi)^{[n]}\\
&+\int_{\hat{X}}\hat{f}\uu(m_{\Phi})dd^{c}\log\big(\frac{e^{\Psi}}{\pi^{*}_X\omega^{n}}\big)\wedge(\pi^{*}_X\omega+dd^{c}\Phi)^{[n]}.
\end{align*}
Notice that the first integral in the first equality is zero: Indeed, if $\Phi=\Phi^\epsilon$ is an $\epsilon$-geodesic, then
\begin{align*}
&\int_{\hat{X}}\hat{f}d\log\big(\frac{e^{\Psi}}{\pi^{*}_X\omega^{n}}\big)\wedge d^c \uu(m_{\Phi^\epsilon})\wedge (\pi^{*}_X\omega+dd^{c}\Phi^{\epsilon})^{[n]}\\
=&\int_{\hat{X}}\hat{f}\langle (d\uu)(m_{\Phi^\epsilon}),(d^c\Psi)_{\tor}-2m_{\Ric(\omega)}\rangle(\pi^{*}_X\omega+dd^{c}\Phi^{\epsilon})^{[n+1]}\\
=&\epsilon\int_{\hat{X}}\hat{f}\langle (d\uu)(m_{\Phi^\epsilon}),(d^c\Psi)_{\tor}-2m_{\Ric(\omega)}\rangle\Big(\pi^{*}_X\omega+\frac{\sqrt{-1}d\tau\wedge d\bar{\tau}}{2|\tau|^{2}}\Big)^{[n+1]},
\end{align*}
since $\Phi^\epsilon\searrow\Phi$ in $C^{1,1}$ topology, passing to the limit as $\epsilon\to 0$, yields
\begin{equation*}
\int_{\hat{X}}\hat{f}d\log\big(\frac{e^{\Psi}}{\pi^{*}_X\omega^{n}}\big)\wedge d^c \uu(m_{\Phi})\wedge (\pi^{*}_X\omega+dd^{c}\Phi)^{[n]}=0.
\end{equation*}
It follows that,
\begin{align*}
\langle dd^{c}\mathcal{H}^{\Psi}_{\uu},f\rangle=&\int_{\hat{X}}\hat{f}\uu(m_{\Phi})dd^{c}\Psi\wedge(\pi^{*}_X\omega+dd^{c}\Phi)^{[n]}\\
&+2\int_{\hat{X}}\hat{f}\uu(m_{\Phi})\pi^{*}_X\Ric(\omega)\wedge(\pi^{*}_X\omega+dd^{c}\Phi)^{[n]}
\end{align*} 
Combining the above equality with \eqref{eq-b} completes the proof.
\end{proof}

Following \cite{BB}, we consider the following modified version of the $(\uu,\vv)$-Mabuchi functional
\begin{equation}\label{Mab-Psi}
\mathcal{M}^{\Psi}_{\uu,\vv}:=\mathcal{H}^{\Psi}_{\uu}-2\mathcal{E}^{\Ric(\omega)}_{\uu}+\mathcal{E}_{\vv}.
\end{equation}
Notice that for $\Psi:=\log(\omega+dd^{c}\phi_\tau)^{n}$ we have $\mathcal{M}^{\Psi}_{\uu,\vv}=\mathcal{M}_{\uu,\vv}$.

\begin{cor}\label{lem-ddcM-weak}
Under the hypothesis of \Cref{ddc-H}, if $\Psi$ is only locally bounded and $dd^c\Psi\geq 0$ as a current, then  
\begin{equation*}
dd^{c}\mathcal{M}^{\Psi}_{\uu}(\phi_\tau)=\int_X\uu(m_{\Phi})dd^{c}\Psi\wedge\big(\pi^{*}_X\omega+dd^{c}\Phi\big)^{[n]},
\end{equation*}
in the weak sense of currents.
\end{cor}

\begin{proof}
Let $\Psi_j$ be a sequence of uniformly bounded, $\mathbb{G}$-invariant smooth functions on $\hat{X}$ such that $\Psi_j\to\Psi$ almost everywhere on $X$ and everywhere on $\mathbb{A}$. Using \Cref{ddc-H} we have 
\begin{equation*}
dd^{c}\mathcal{M}^{\Psi_j}_{\uu}(\phi_\tau)=\int_X\uu(m_{\Phi})dd^{c}\Psi_j\wedge\big(\pi^{*}_X\omega+dd^{c}\Phi\big)^{[n]}.
\end{equation*}
By the dominated convergence theorem (notice that $\uu(m_\Phi)$ is uniformly bounded), we can pass to the limit when $j\to \infty$ (see e.g. \cite[Proposition 3.2]{Demaily}).
\end{proof}

Now, we can use the arguments of Berman--Berndtsson in \cite{BB} to deduce the weak convexity of the $(\uu,\vv)$-Mabuchi energy along weak geodesic segments. We will need the following regularization result which is the main ingredient in the proof of Berman--Berndtsson for the weak convexity of the Mabuchi energy \cite[Theorem 3.3]{BB}.
\begin{prop}[\cite{BB}]\label{prop-BB}
Let $(\phi_t)_{t\in[0,1]}\in\mathcal{K}^{1,1}(X,\omega)^{\T}$ be a weak geodesic segment, and $\Phi$ the corresponding weak solution of \eqref{MA}. Let $\Psi:=\log(\pi^{*}_X\omega+dd^{c}\Phi)^{n}$.
\begin{enumerate}
\item\label{prop-BB-i} There exist a family of locally bounded $\mathbb{G}$-invariant functions $(\Psi_A)_{A>0}$ on $\hat{X}$, such that $dd^{c}\Psi_A\geq 0$ in the weak sens of currents, and $\Psi_A\searrow\Psi$ as $A\to\infty$.
\item\label{prop-BB-ii} For fixed $A>0$, there exist a family of $\mathbb{G}$-invariant functions $(\Psi_{k,A})_{A>0}$ on $\hat{X}$ with continuous dependence on $\tau\in\mathbb{A}$, such that the currents $T_{A,k}:=dd^{c}\Psi_{k,A}\wedge(\pi^{*}_X\omega+dd^{c}\Phi)^{n}$ are positive and $\Psi_{k,A}\to\Psi_A$ pointwise almost everywhere on $X$ and everywhere on $\tau$ as $k\to\infty$.
\end{enumerate}
\end{prop}
Using the above proposition together with \Cref{lem-ddcM-weak}, we get the following 
\begin{thm}\label{Mab-conv}
Let $(\phi_t)_{t\in[0,1]}\in\mathcal{K}^{1,1}(X,\omega)^{\T}$ be a weak geodesic segment. The function $\tau\mapsto\mathcal{M}_{\uu,\vv}(\phi_\tau)$ is weakly subharmonic on $\mathbb{A}$. In particular, $\mathcal{M}_{\uu,\vv}(\phi_t)$ is weakly convex along the weak geodesic $(\phi_t)$. 
\end{thm}

\begin{proof}
By \Cref{lem-ddcM-weak}, since the function $\Psi_A$ from \ref{prop-BB-i} in \Cref{prop-BB} is locally bounded, we obtain
\begin{equation}
dd^{c}\mathcal{M}_{\uu,\vv}^{\Psi_A}(\phi_\tau)=\int_X u(m_\Phi)T_A.
\end{equation}
where $T_A:=dd^{c}\Psi_{A}\wedge(\pi^{*}_X\omega+dd^{c}\Phi)^{n}$. Now using the fact that $\uu(m_\Phi)T_{A,k}\geq 0$ are positive Radon measures which converges weakly to $\uu(m_\Phi)T_{A}$ as $k\to \infty$. It follows that $dd^{c}\mathcal{M}_{\uu,\vv}^{\Psi_A}(\phi_\tau)\geq 0$. On the other hand we have $\mathcal{M}_{\uu,\vv}^{\Psi_A}(\phi_\tau)\to\mathcal{M}_{\uu,\vv}(\phi_\tau)$ as $A\to\infty$. Thus, $dd^{c}\mathcal{M}_{\uu,\vv}(\phi_\tau)\geq 0$ in the weak sens of currents.
\end{proof}
To get the pointwise convexity of $t\mapsto\mathcal{M}_{\uu,\vv}(\phi_t)$, we have to show that it is continuoues. For the energy part $t\mapsto-2\mathcal{E}^{\Ric(\omega)}_{\uu}(\phi_t)+\mathcal{E}_{\vv}(\phi_t)$, it is clear from \eqref{ext-E-v} and \eqref{ext-E-u-theta} that it is a continuous function, since $t\to\phi_t$ is a continuous family. As in the case when $\uu \equiv 1$ on $\PP_\alpha$ (see \cite{BB}), it is not a priori clear that the entropy part $t\mapsto\mathcal{H}_{\uu}(\phi_t)$ is continuous. 
\begin{thm}
The $(\uu,\vv)$-Mabuchi energy $\mathcal{M}_{\uu,\vv}$ is continuous along weak geodesics and therefore convex in the pointwise sense.
\end{thm}

\begin{proof}
The argument is very similar to the one of Berman--Berndtsson in \cite[Theorem 3.4]{BB}, the only difference is in the calculation of the second variation of the entropy term involving the weighted measure $\uu(m_{\phi_\tau})\omega_{\phi_\tau}^{[n]}$.

Let $\kappa_\epsilon(s)$ be a sequence of strictly convex functions such that $\kappa^{\prime}_\epsilon(s)\geq 1$ and $\kappa_\epsilon(s)\to s$ as $\epsilon\to 0$. Let $\zeta_j$ be a partition of unity subordinate to an open cover of $X$. We consider the following modification of the entropy term
\begin{equation*}
\mathcal{H}_{\uu,j,\epsilon}^{\Psi_{A,k}}(\phi_\tau)=\int_X \zeta_j\kappa_\epsilon\big(\log\big(\frac{e^{\Psi_{A,k}(\cdot,\tau)}}{\omega^{n}}\big)\big)\uu(m_{\phi_\tau})\omega_{\phi_\tau}^{[n]},
\end{equation*}
where $\Psi_{A,k}$ is given in \Cref{prop-BB} \ref{prop-BB-ii} (see also \cite[Theorem 3.3]{BB} for more details). From the calculations in the proof of \Cref{ddc-H} we have
\begin{align*}
dd^{c}\mathcal{H}_{\uu,j,\epsilon}^{\Psi_{A,k}}&(\phi_\tau)=\int_X \zeta_j\uu(m_{\phi_\tau})dd^{c}\kappa_\epsilon\big(\log\big(\frac{e^{\Psi_{A,k}}}{\omega^{n}}\big)\big)\wedge(\pi^{*}_X \omega+dd^c\Phi)^{[n]}\\
=&\int_X \zeta_j\uu(m_{\phi_\tau})d\Big(\kappa^{\prime}_\epsilon\big(\log\big(\frac{e^{\Psi_{A,k}}}{\omega^{n}}\big)\big)d^{c}\log\big(\frac{e^{\Psi_{A,k}}}{\omega^{n}}\big)\Big)\wedge(\pi^{*}_X \omega+dd^c\Phi)^{[n]}\\
=&\int_X \zeta_j\uu(m_{\phi_\tau})\kappa^{\prime}_\epsilon\big(\log\big(\frac{e^{\Psi_{A,k}}}{\omega^{n}}\big)\big)dd^{c}\log\big(\frac{e^{\Psi_{A,k}}}{\omega^{n}}\big)\wedge(\pi^{*}_X \omega+dd^c\Phi)^{[n]}\\
&+\int_X \zeta_j\uu(m_{\phi_\tau})\kappa^{\prime\prime}_\epsilon\big(\log\big(\frac{e^{\Psi_{A,k}}}{\omega^{n}}\big)\big)d\log\big(\frac{e^{\Psi_{A,k}}}{\omega^{n}}\big)\wedge d^{c}\log\big(\frac{e^{\Psi_{A,k}}}{\omega^{n}}\big)\wedge(\pi^{*}_X \omega+dd^c\Phi)^{[n]},
\end{align*}
It follows that,
\begin{align}
\begin{split}\label{ddc-H-j-k-A}
dd^{c}\mathcal{H}_{\uu,j,\epsilon}^{\Psi_{A,k}}&(\phi_\tau)=\int_X \zeta_j\uu(m_{\phi_\tau})\kappa^{\prime}_\epsilon\big(\log\big(\frac{e^{\Psi_{A,k}}}{\omega^{n}}\big)\big)T_{A,k}\\
&+2\int_X \zeta_j\uu(m_{\phi_\tau})\kappa^{\prime}_\epsilon\big(\log\big(\frac{e^{\Psi_{A,k}}}{\omega^{n}}\big)\big)\pi^{*}_X\Ric(\omega)\wedge(\pi^{*}_X \omega+dd^c\Phi)^{[n]}\\
&+\int_X \zeta_j\uu(m_{\phi_\tau})\kappa^{\prime\prime}_\epsilon\big(\log\big(\frac{e^{\Psi_{A,k}}}{\omega^{n}}\big)\big)d\log\big(\frac{e^{\Psi_{A,k}}}{\omega^{n}}\big)\wedge d^{c}\log\big(\frac{e^{\Psi_{A,k}}}{\omega^{n}}\big)\wedge(\pi^{*}_X \omega+dd^c\Phi)^{[n]},
\end{split}
\end{align} 
where $T_{A,k}:=dd^{c}\Psi_{A,k}\wedge(\pi^{*}_X \omega+dd^c\Phi)^{[n]}$. Now we introduce the following modified version of the $(\uu,\vv)$-Mabuchi energy:
\begin{equation*}
\mathcal{M}_{\uu,\vv,j,\epsilon}^{\Psi_{A,k}}:=\mathcal{H}_{\uu,j,\epsilon}^{\Psi_{A,k}}-2\mathcal{E}_{\uu,j}^{\theta_j},
\end{equation*}
where $\theta_j:=\zeta_j\Ric(\omega)$.  Combining \eqref{ddc-H-j-k-A} with \eqref{eq-a} and \eqref{eq-b}, we obtain 
\begin{align*}
dd^c\mathcal{M}_{\uu,\vv,j,\epsilon}^{\Psi_{A,k}}&(\phi_\tau)=\int_X \zeta_j\uu(m_{\phi_\tau})\kappa^{\prime}_\epsilon\big(\log\big(\frac{e^{\Psi_{A,k}}}{\omega^{n}}\big)\big)T_{A,k}\\
&+2\int_X\Big[1-\kappa^{\prime}_\epsilon\big(\log\big(\frac{e^{\Psi_{A,k}}}{\omega^{n}}\big)\big)\Big] \zeta_j\uu(m_{\phi_\tau})\pi^{*}_X\Ric(\omega)\wedge(\pi^{*}_X \omega+dd^c\Phi)^{[n]}\\
&+\int_X \zeta_j\uu(m_{\phi_\tau})\kappa^{\prime\prime}_\epsilon\big(\log\big(\frac{e^{\Psi_{A,k}}}{\omega^{n}}\big)\big)d\log\big(\frac{e^{\Psi_{A,k}}}{\omega^{n}}\big)\wedge d^{c}\log\big(\frac{e^{\Psi_{A,k}}}{\omega^{n}}\big)\wedge(\pi^{*}_X \omega+dd^c\Phi)^{[n]}.
\end{align*}
Since $\kappa_{\epsilon}$ is strictly convex, the integral in the last line is positive, and using that $\kappa^{\prime}_\epsilon(s)\geq 1$, together with $T_{A,k}\geq 0$, it is also clear that the integral in the first line is positive. For the remaining integral we can bound it from below by $-C_{\epsilon,j}\frac{\sqrt{-1}d\tau\wedge d\bar{\tau}}{2|\tau|^2}$ for some $C_{\epsilon,j}\geq 0$. Thus,
\begin{equation*}
dd^c\mathcal{M}_{\uu,\vv,j,\epsilon}^{\Psi_{A,k}}(\phi_\tau)\geq-C_{\epsilon,j}\frac{\sqrt{-1}d\tau\wedge d\bar{\tau}}{2|\tau|^2}.
\end{equation*} 
It follows that the function $t\mapsto\mathcal{M}_{\uu,\vv,j,\epsilon}^{\Psi_{A,k}}(\phi_\tau)+C_{\epsilon,j}t^{2}$ (where $\tau=e^{-t+is}$) is weakly convex. On the other hand $\tau \mapsto\mathcal{M}_{\uu,\vv,j,\epsilon}^{\Psi_{A,k}}(\phi_\tau)$ is continuous since $\Psi_{A,k}$ is continuous in $\tau\in\mathbb{A}$, by \Cref{prop-BB} \ref{prop-BB-ii}. It follows that $t\mapsto\mathcal{M}_{\uu,\vv,j,\epsilon}^{\Psi_{A,k}}(\phi_t)+C_{\epsilon,j}t^{2}$ is convex in the pointwise sense. Using the equation, 
\begin{equation*}
\mathcal{M}_{\uu,\vv}(\phi_t)-\mathcal{E}_{\vv}(\phi_t)=\underset{\epsilon\to 0}{\lim}\sum_j \mathcal{M}_{\uu,\vv,j,\epsilon}^{\Psi_{A,k}}(\phi_\tau) +C_\epsilon t^{2},
\end{equation*}
where $C_\epsilon=\sum_j C_{j,\epsilon}$, we infer that $t\mapsto\mathcal{M}_{\uu,\vv}(\phi_t)-\mathcal{E}_{\vv}(\phi_t)$ is convex in the pointwise sense, thus continuous. By \eqref{Entropy}, the function $t\to \mathcal{H}_{\uu}(\phi_t)$ is lower semicontinuous, then it is continuous on $[0,1]$. This, completes the proof.
\end{proof}
\section{Proof of \Cref{thm-Mab-bound}}
\begin{lem}\label{lem-Ent}
Given a weak geodesic segment $(\phi_t)_{t\in[0,1]}\in\mathcal{K}^{1,1}(X,\omega)^{\T}$ connecting $\phi_0,\phi_1\in\mathcal{K}(X,\omega)^{\T}$, we have the following inequalities
\begin{align*}
\underset{t\rightarrow 0^{+}}{\lim}\dfrac{\mathcal{H}_{\uu}(\phi_t)-\mathcal{H}_{\uu}(\phi_0)}{t}\geq& -\int_X \tilde{\uu}(m_{\phi_0})\dot{\phi}\left(\Ric(\omega_{\phi_0})-\Ric(\omega)\right)\wedge \omega^{[n-1]}_{\phi_0}\\
&-\int_X\langle (d\tilde{\uu})(m_{\phi_0}),m_{\Ric(\omega_{\phi_0})}-m_{\Ric(\omega)}\rangle\dot{\phi}\omega_{\phi_0}^{[n]}\\
&-\int_X\dot{\phi}\Delta_{\phi_0}(\tilde{\uu}(m_{\phi_0}))\omega^{[n]}_{\phi_0}
\end{align*}
where $\dot{\phi}=\left.\frac{d\phi_t}{dt}\right|_{t=0^{+}}$ and $\tilde{\uu}:=\frac{\uu}{\vol(X,\uu(m_\omega)\omega^{[n]})}$. 
\end{lem}

\begin{proof}
By convexity of the entropy with respect to the affine structure on the space of probability measures (see e.g. \cite{BB, DZ}) and using \eqref{Entropy}, we have  
\begin{align*}
&\dfrac{\mathcal{H}_{\uu}(\phi_t)-\mathcal{H}_{\uu}(\phi_0)}{t}=\dfrac{\Ent_{\mu_\omega}(\mu_{\uu}(\phi_t))-\Ent_{\mu_\omega}(\mu_{\uu}(\phi_0))}{t}\\
\geq&\int_X\log\Big(\frac{\mu_{\uu}(\phi_0)}{\mu_\omega}\Big)\dfrac{\mu_{\uu}(\phi_t)-\mu_{\uu}(\phi_0)}{t}\\
=&\int_X\log\Big(\frac{\omega^{n}_{\phi_0}}{\mu_\omega}\Big)\dfrac{\mu_{\uu}(\phi_t)-\mu_{\uu}(\phi_0)}{t}+\int_X\log(\tilde{\uu}(m_{\phi_0}))\dfrac{\mu_{\uu}(\phi_t)-\mu_{\uu}(\phi_0)}{t}\\
=&\int_X\log\Big(\frac{\omega^{n}_{\phi_0}}{\mu_\omega}\Big)\dfrac{\mu_{\uu}(\phi_t)-\mu_{\uu}(\phi_0)}{t}+\int_X\frac{1}{t}\left(\log(\tilde{\uu}(m_{\phi_0}))\tilde{\uu}(m_{\phi_t})\omega^{[n]}_{\phi_t}-\log(\tilde{\uu}(m_{\phi_0}))\tilde{\uu}(m_{\phi_0})\omega^{[n]}_{\phi_0}\right)\\
=&\int_X\log\Big(\frac{\omega^{n}_{\phi_0}}{\mu_\omega}\Big)\dfrac{\mu_{\uu}(\phi_t)-\mu_{\uu}(\phi_0)}{t}+\int_X\frac{1}{t}\left(\log(\tilde{\uu}(m_{\phi_0}))-\log(\tilde{\uu}(m_{\phi_t}))\right)\tilde{\uu}(m_{\phi_t})\omega^{[n]}_{\phi_t}
\end{align*}
where we have used the fact that $$\int_X\log(\tilde{\uu}(m_{\phi_t}))\tilde{\uu}(m_{\phi_t})\omega^{[n]}_{\phi_t}=\int_X\log(\tilde{\uu}(m_{\phi_0}))\tilde{\uu}(m_{\phi_0})\omega^{[n]}_{\phi_0}={\rm const}$$ 
is a constant independent of $t$. We thus compute
\begin{align*}
&\underset{t\rightarrow 0^{+}}{\lim}\dfrac{\mathcal{H}_{\uu}(\phi_t)-\mathcal{H}_{\uu}(\phi_0)}{t}\\
\geq&-\int_X\tilde{\uu}(m_{\phi_0})d\dot{\phi}\wedge d^{c}\log\Big(\frac{\omega^{n}_{\phi_0}}{\omega^{n}}\Big)\wedge \omega^{[n-1]}_{\phi_0}+\int_X\dot{\phi}dd^{c}(\tilde{\uu}(m_{\phi_0}))\wedge\omega^{[n-1]}_{\phi_0}\\
=&\int_X\dot{\phi}\Big(d(\tilde{\uu}(m_{\phi_0})),d\log\Big(\frac{\omega^{[n]}_{\phi_0}}{\omega^n}\Big)\Big)_{\phi_0}\omega^{[n]}_{\phi_0}+\int_X\tilde{\uu}(m_{\phi_0})\dot{\phi}dd^{c}\log\Big(\frac{\omega^{n}_{\phi_0}}{\omega^{n}}\Big)\wedge \omega^{[n-1]}_{\phi_0}\\&+\int_X\dot{\phi}dd^{c}(\tilde{\uu}(m_{\phi_0}))\wedge\omega^{[n-1]}_{\phi_0}   \\
=&-\int_X\tilde{\uu}(m_{\phi_0})\dot{\phi}\left(\Ric(\omega_{\phi_0})-\Ric(\omega)\right)\wedge \omega^{[n-1]}_{\phi_0}-\int_X\langle (d\tilde{\uu})(m_{\phi_0}),m_{\Ric(\omega_{\phi_0})}-m_{\Ric(\omega)}\rangle\dot{\phi}\omega_{\phi_0}^{[n]}\\
&-\int_X\dot{\phi}\Delta_{\phi_0}(\tilde{\uu}(m_{\phi_0}))\omega^{[n]}_{\phi_0}.
\end{align*}
\end{proof}
Now we are in position to give the proof of \Cref{thm-Mab-bound}.
\begin{proof}[Proof of \Cref{thm-Mab-bound}]
Let $(\phi_t)_{t\in[0,1]}\in\mathcal{K}^{1,1}(X,\omega)^{\T}$ be a weak geodesic segment connecting $\phi_0,\phi_1\in\mathcal{K}(X,\omega)^{\T}$. We suppose that $\tilde{\uu}:=\frac{\uu}{\vol(X,\uu(m_\omega)\omega^{[n]})}=\uu$. We have
\begin{align*}
&\underset{t\rightarrow 0^{+}}{\lim}\dfrac{\mathcal{E}_{\vv}(\phi_t)-\mathcal{E}_{\vv}(\phi_0)}{t}=\int_X\dot{\phi}\vv(m_{\phi_0})\omega_{\phi_0}^{[n]}\\
&\underset{t\rightarrow 0^{+}}{\lim}\dfrac{\mathcal{E}^{\Ric(\omega)}_{\uu}(\phi_t)-\mathcal{E}^{\Ric(\omega)}_{\uu}(\phi_0)}{t}=\int_X\dot{\phi}\big(\uu(m_{\phi_0})\Ric(\omega)\wedge\omega^{[n-1]}_{\phi_0}+\langle(d\uu)(m_{\phi_0}),m_{\Ric(\omega)}\rangle\omega^{[n]}_{\phi_0}\big)
\end{align*} 
By \Cref{lem-Ent} and \Cref{Chen-Tian} we get
\begin{align*}
\underset{t\rightarrow 0^{+}}{\lim}\dfrac{\mathcal{M}_{\uu,\vv}(\phi_t)-\mathcal{M}_{\uu,\vv}(\phi_0)}{t}\geq&\int_X(-\Scal_{\uu}(\phi_0)+\vv(m_{\phi_0}))\dot{\phi}\omega_{\phi_0}^{[n]}.
\end{align*}
Using the sub-slop inequality for the convex function $\mathcal{M}_{\uu,\vv}(\phi_t)$ we get
\begin{align*}
\mathcal{M}_{\uu,\vv}(\phi_1)-\mathcal{M}_{\uu,\vv}(\phi_0)\geq& \underset{t\rightarrow 0^{+}}{\lim}\dfrac{\mathcal{M}_{\uu,\vv}(\phi_t)-\mathcal{M}_{\uu,\vv}(\phi_0)}{t}\\
\geq&\int_X(-\Scal_{\uu}(\phi_0)+\vv(m_{\phi_0}))\dot{\phi}\omega_{\phi_0}^{[n]}.
\end{align*} 
By Cauchy-Schwartz inequality we obtain
\begin{equation*}
\mathcal{M}_{\uu,\vv}(\phi_1)-\mathcal{M}_{\uu,\vv}(\phi_0)\geq -d(\phi_1,\phi_0)\parallel\Scal_{\uu}(\phi_0)-\vv(m_{\phi_0})\parallel_{L^{2}(X,\mu_{\phi_0})}.
\end{equation*} 
For the general case where $\tilde{\uu}\neq\uu$, we apply the above formula to the $(\tilde{\uu},\frac{\vv}{\vol(X,\uu(m_\omega)\omega^{[n]})})$-Mabuchi energy.
\end{proof}

\section{Uniqueness of weighted cscK metrics}
This section is devoted to establish \Cref{Thm-uniq} from the introduction. We will generalise the approach of \cite{BB, CPZ} to the weighted setting. Our proof is closer to the method used by Chen--Paun--Zeng \cite{CPZ}, based on a generalisation of the bifurcation technique of Bando--Mabuchi \cite{BM}.
\begin{prop}\label{Thm-Def}
Let $X$ be a compact K\"ahler manifold with K\"ahler class $\alpha$, $\T\subset\Autred$ a real torus with momentum polytope $\PP_\alpha\subset \tor^*$ and $\uu\in C^\infty(\PP_\alpha,\mathbb{R}_{>0})$, and $\vv\in C^\infty(\PP_\alpha,\mathbb{R}_{>0})$ a non vanishing function on $\PP_\alpha$. If $\omega\in\alpha$ is a $\T$-invariant K\"ahler metric, and $\varphi_0\in \mathcal{K}(X,\omega)^{\T}$ such that $\omega_{\varphi_0}\in\alpha$ a $(\uu,\vv)$-extremal metric. Then, there exist $\omega_{\phi_0}$ in the orbit of $\omega_{\varphi_0}$ under the action of the group $G:={\rm Aut}_{\rm red}^{\T}(X)^{\circ}$, and a smooth function $\phi:[0,\epsilon)\times X\to \mathbb{R}$, such that $\phi_t:=\phi(t,\cdot)\in\mathcal{K}(X,\omega)^{\T}$ satisfies the equation
\begin{equation}\label{eq-uu-vv-ext-twist}
\Scal_{\uu}(\phi_t)-t\big(\uu(m_{\phi_t})\Lambda_{\phi_t}\omega+\big\langle(d\uu)(m_{\phi_t}),m_\omega\big\rangle\big)=\ell_{\rm ext}(m_{\phi_t})\vv(m_{\phi_t}),
\end{equation}
where $\Lambda_{\phi}\omega$ is the trace of $\omega$ with respect to $\omega_\phi$ and $\ell_{\rm ext}$ is the $(\uu,\vv)$-extremal affine linear function of $(\alpha,\T,\uu,\vv)$.
\end{prop}
The proof follows from an application of the inverse function theorem as in \cite{CPZ}. To this end we need to find the K\"ahler metric $\omega_{\phi_0}$ in the $G$-orbit of the $(\uu,\ell_{\rm ext}\cdot\vv)$ metric $\omega_{\varphi_0}$, as stated in the theorem. 

Let $\hat{\mathcal{K}}(X,\omega)^{\T}$ denote the space of $\T$-invariant K\"ahler potentials $\phi\in\mathcal{K}(X,\omega)^{\T}$ normalized by $\int_X\phi\vv(m_\omega)\omega^{[n]}=0$, and $K:={\rm Isom}^{\T}_0(X,\omega_{\varphi_0})\cap G$ the connected component of identity of the group of Hamiltonian isometries of $(X,\omega_{\varphi_0})$ commuting with $\T$. As we suppose by definition that $\Scal_\uu(\varphi_0)/\vv(m_{\varphi_0}) =\ell_{\rm ext}(m_{\varphi_0})$ is the Killing potential of a vector field in $\tor$, by \cite[Corollary B.1]{lahdili3}  $K$ is a maximal connected compact subgroup of $G$. Following \cite{CPZ}, we consider the map 
\begin{equation*}
\Psi^{\omega}:\mathcal{O}\to\hat{\mathcal{K}}(X,\omega)^{\T},
\end{equation*}
defined on the homogeneous manifold $\mathcal{O}:=G/K$  by $\Psi^{\omega}(\sigma):=\phi_\sigma$, where $\phi_\sigma\in\hat{\mathcal{K}}(X,\omega)^{\T}$ is the unique potential such that 
\begin{equation}\label{phi-sigma}
\sigma^*\omega=\omega+dd^c\phi_\sigma.
\end{equation}
In the case when $\omega$ is $(\uu,\vv)$-extremal metric, \cite[Theorem B.1]{lahdili3} yields the following result, which is a straightforward generalization of \cite[Proposition 4.3]{CPZ} describing $(T_\sigma\Psi^\omega)(T_\sigma\mathcal{O})$ the image of the differential of $\Psi^\omega$ in $\sigma\in\mathcal{O}$.
\begin{lem}\cite[Proposition 4.3]{CPZ}\label{lem-TPsi}
If $\omega$ is a $(\uu,\vv)$-extremal metric, then the image $(T_\sigma\Psi^\omega)(T_\sigma\mathcal{O})$ is given by real holomorphic vector fields 
$$\xi=J\grad_{\phi_\sigma}(f)\in \mathfrak{k},$$
where $\mathfrak{k}:={\rm Lie}(K)$, $\phi_\sigma=\Psi^\omega(\sigma)$ and $f\in C^\infty(X,\mathbb{R})$.
\end{lem}
By a result due to Mabuchi \cite{Mabuchi2}, any real holomorphic vector field $\xi\in (T_\sigma\Psi^\omega)(T_\sigma\mathcal{O})$, gives rise to a smooth geodesic ray $(\phi_t)_{t\in\mathbb{R}}\in\mathcal{K}(X,\omega)^{\T}$, defined by $\phi_t:=\Psi^\omega(\exp(t\xi))$. Using, strict convexity of the functional $\mathcal{E}^{\theta}_{\uu}$ along weak geodesics (see \Cref{Cor-E-theta-strict}) and the fact that $\exp:T\mathcal{O}\to\mathcal{O}$ is onto, we obtain 
\begin{lem}\cite[Lemma 2]{CPZ}\label{lem-min}
If $\omega$ is a $(\uu,\vv)$-extremal metric, then for any $\T$-invariant K\"ahler form $\theta$ on $X$, the functional $\mathcal{E}^{\theta}_{\uu}\circ\Psi^\omega:\mathcal{O}\to \mathbb{R}$ is proper. In particular $\mathcal{E}^{\theta}_{\uu}$ admits a unique minimum point on the orbit $\Psi^\omega(\mathcal{O})$.
\end{lem}
Now we are in position to give a sketch for the proof of \Cref{Thm-Def}, which is not materially different than \cite[Theorem 1.2]{CPZ}.

\begin{proof}[Proof of \Cref{Thm-Def}]
Since $\omega_{\varphi_0}$ is a $(\uu,\vv)$-extremal metric, we can take $\phi_0\in\Psi^{\omega_{\varphi_0}}(\mathcal{O})$ be the unique minimiser of $\mathcal{E}^{\omega}_{\uu}$ (we take $\theta=\omega$ in \Cref{lem-min}). Using \Cref{lem-TPsi} and \eqref{E-theta}, we have
\begin{equation}\label{orth}
\left\langle\vv(m_{\phi_0})^{-1}(\uu(m_{\phi_0})\Lambda_{\phi_0}\omega+\langle(d\uu)(m_{\phi_0}),m_\omega\rangle),f\right\rangle_{\vv,\phi_0}=0,
\end{equation}
for any $f\in \mathfrak{k}_{\phi_0}$ in the space of $\omega_{\phi_0}$-Killing potentials of elements of ${\rm Lie}(K):=\mathfrak{k}$, where $\langle\cdot,\cdot\rangle_{\vv,\phi_0}$ is the weighted inner product 
\begin{equation}\label{prod}
\langle f,h\rangle_{\vv,\phi_0}=\int_X fh\vv(m_{\phi_0})\omega_{\phi_0}^{[n]}.
\end{equation}
Let $\mathcal{K}^{2,k+4}(X,\omega)^{\T}$ be the open set of $\T$-invariant $\omega$-K\"ahler potentials with ${\rm L}^{2,k+4}$ regularity. We consider the map:
\begin{equation*}
\mathcal{F}_{\uu,\vv}:\mathcal{K}^{2,k+4}(X,\omega)^{\T}\times [0,1]\to{\rm L}^{2,k}(X,\mathbb{R})^{\T}\times [0,1],
\end{equation*}
defined by
\begin{align}
\begin{split}\label{F-u-v}
&\mathcal{F}_{\uu,\vv}(\phi,t):=(F_{\uu,\vv}(\phi,t),t),\\
&F_{\uu,\vv}(\phi,t):=\frac{\Scal_{\uu}(\phi)-t\big(\uu(m_{\phi})\Lambda_{\phi}(\omega)+\big\langle(d\uu)(m_{\phi}),m_\omega\big\rangle\big)}{\vv(m_{\phi})}-\ell_{\rm ext}(m_{\phi}).
\end{split}
\end{align}
We have $\mathcal{F}_{\uu,\vv}(\phi_0,0)=0$. Using \cite[Lemma B.1]{lahdili3}, we can calculate the differential at $(\phi_0,0)$ of $\mathcal{F}_{\uu,\vv}$ is given by 
\begin{align*}
T_{(\phi_0,0)}\mathcal{F}_{\uu,\vv}:&{\rm L}^{2,k+4}(X,\mathbb{R})^{\T}\times \mathbb{R}\to{\rm L}^{2,k}(X,\mathbb{R})^{\T}\times \mathbb{R},\\
(T_{(\phi_0,0)}\mathcal{F}_{\uu,\vv})(\dot{\phi},\zeta)=&\Big((T_{(\phi_0,0)}F_{\uu,\vv})(\dot{\phi},\zeta),\zeta\Big),\\
(T_{(\phi_0,0)}F_{\uu,\vv})(\dot{\phi},\zeta)=&-\frac{\mathcal{D}_{\phi_0}^{*}\uu(m_{\phi_0})\mathcal{D}_{\phi_0}\dot{\phi}}{\vv(m_{\phi_0})}-\zeta\left[\frac{\uu(m_{\phi_0})\Lambda_{\phi_0}\omega+\big\langle(d\uu)(m_{\phi_0}),m_\omega\big\rangle}{\vv(m_{\phi_0})}\right],
\end{align*}
where $\mathcal{D}_{\phi_0}\dot{\phi}:=\sqrt{2}(\nabla^{\phi_0} d\dot{\phi})^{-}$ is the $J$-anti-invariant part of the tensor $\nabla^{\phi_0} d\dot{\phi}$, with $\nabla^{\phi_0}$ the $g_{\phi_0}$-Levi-Civita connection, and $\mathcal{D}_{\phi_0}^{*}$ is the formal adjoint of $\mathcal{D}_{\phi_0}$. 

Notice that $\mathbb{L}_{\uu,\vv}:=(\vv(m_{\phi_0}))^{-1}\mathcal{D}_{\phi_0}^{*}\uu(m_{\phi_0})\mathcal{D}_{\phi_0}$ is a fourth order $\langle\cdot,\cdot\rangle_{\vv,\phi_0}$-self adjoint $\T$-invariant elliptic linear operator. By standard elliptic theory we have the following decomposition $\langle\cdot,\cdot\rangle_{\vv,\phi_0}$-orthogonal decomposition 
\begin{equation}\label{decomp}
{\rm L}^{2,k}(X,\R)^{\T}={\rm Ker}(\mathbb{L}_{\uu,\vv})\oplus{\rm Im}(\mathbb{L}_{\uu,\vv}).
\end{equation}
We have ${\rm Ker}(\mathbb{L}_{\uu,\vv})=\mathfrak{k}_{\phi_0}$ since $K$ is a maximal compact subgroup of $G$, and ${\rm Im}(\mathbb{L}_{\uu,\vv})={\rm L}^{2,k}_{\perp}(X,\mathbb{R})^{\T}$ . Using \eqref{decomp}, it's clear that the linearization is neither injective nor surjective. Let $\Pi_{\vv,\phi_0}$ the $\langle\cdot,\cdot\rangle_{\vv,\phi_0}$-orthogonal projection on $\mathfrak{k}_{\phi_0}$. 

We consider the following modification of the map $\mathcal{F}_{\uu,\vv}$ 
\begin{align*}
\tilde{\mathcal{F}}_{\uu,\vv}:\mathcal{K}^{2,k+4}(X,\omega)^{\T}\times [0,1]\to\mathfrak{k}_{\phi_0}\times{\rm L}^{2,k}_{\perp}(X,\mathbb{R})^{\T}\times [0,1]
\end{align*}
defined by 
\begin{equation*}
\tilde{\mathcal{F}}_{\uu,\vv}(f,\psi,t):=(f,(I-\Pi_{\vv,\phi_0})\circ F_{\uu,\vv}(f+\psi,t),t).
\end{equation*}
where $f\in\mathfrak{k}_{\phi_0}$ and $\psi\in{\rm L}^{2,k}_{\perp}(X,\mathbb{R})^{\T}$ such that $\phi:=f+\psi\in\mathcal{K}^{2,k+4}(X,\omega)^{\T}$. Let $\phi_0:=f_0+\psi_0$ be the orthogonal decomposition of $\phi_0$ in \eqref{decomp}. The derivative of $\tilde{\mathcal{F}}_{\uu,\vv}$ in $(f_0,\psi_0,0)$, is given by
\begin{equation*}
(T_{(f_0,\psi_0,0)}\tilde{\mathcal{F}}_{\uu,\vv})(f,\dot{\psi},\zeta)=\Big(f,-\mathbb{L}_{\uu,\vv}(\dot{\psi})-\zeta\vv(m_{\phi_0})^{-1}\big(\uu(m_{\phi_0})\Lambda_{\phi_0}\omega+\big\langle(d\uu)(m_{\phi_0}),m_\omega\big\rangle\big),\zeta\Big).
\end{equation*}
The decomposition \eqref{decomp} and the equation \eqref{orth} show that $T_{(f_0,\psi_0,0)}\tilde{\mathcal{F}}_{\uu,\vv}$ is bijective. By the inverse function theorem we obtain a path 
\begin{equation}\label{phi-t-f}
\phi(f,t):=f+\psi(f,t)\in\mathcal{K}^{2,k+4}(X,\omega)^{\T},
\end{equation}
for $0<t<\epsilon$ and $f\in\mathfrak{k}_{\phi_0}$, such that
\begin{equation}\label{sol-pi-2}
(I-\Pi_{\vv,\phi_0})\circ F_{\uu,\vv}(\phi(f,t),t)=0
\end{equation}
for $\parallel f-f_0\parallel_{{\rm L}^{2,k+4}}<\epsilon$. 

Now we introduce the functional $\mathcal{G}_{\uu,\vv}:\mathfrak{k}_{\phi_0}\times(0,\epsilon)\to \mathfrak{k}_{\phi_0}$, defined by
\begin{equation*}
\mathcal{G}_{\uu,\vv}(f,t):=\Pi_{\vv,\phi_0}\circ F_{\uu,\vv}(\phi(f,t),t),
\end{equation*}
where $\phi(f,t)$ is given by \eqref{phi-t-f}. To complete the proof we need to solve the equation
\begin{equation*}
\mathcal{G}_{\uu,\vv}(f(t),t)=0,
\end{equation*}
for $t\in(0,\epsilon)$ and $f(t)\in\mathfrak{k}_{\phi_0}$. However, its not possible to apply the implicit function theorem. Indeed,
\begin{equation*}
\left.\frac{\partial \mathcal{G}_{\uu,\vv}}{\partial f}\right|_{(f_0,0)}(\dot{f})=\Pi_{\vv,\phi_0}\circ \left.\frac{\partial F_{\uu,\vv}}{\partial f}\right|_{(f_0,0)}\Big(\dot{f}+\left.\frac{\partial \psi}{\partial f}\right|_{(f_0,0)}(\dot{f})\Big)=0
\end{equation*}
since, by differentiating \eqref{sol-pi-2} with respect to $f$, we get 
\begin{equation}\label{d-phi-d-f}
\left.\frac{\partial \psi}{\partial f}\right|_{(f_0,0)}(\dot{f})=0.
\end{equation}
To solve this problem, one can consider the map \cite{CPZ}:
\begin{equation*}
\tilde{\mathcal{G}}_{\uu,\vv}(f,t):=\begin{cases}
\frac{\mathcal{G}_{\uu,\vv}(f,t)}{t}\quad&\text{if }t\neq 0,\\
\left.\frac{\partial\mathcal{G}_{\uu,\vv}}{\partial t}\right|_{(f,0)}\quad&\text{if }t=0.
\end{cases}
\end{equation*}
which is continuous on $\mathfrak{k}_{\phi_0}\times[0,1]$. We want to apply the implicit function theorem to solve the equation 
$$\tilde{\mathcal{G}}_{\uu,\vv}(f(t),t)=0.$$
So we have to check that the derivative
\begin{equation}\label{A-map}
Q_{\uu,\vv}:=\left.\frac{\partial \tilde{\mathcal{G}}_{\uu,\vv}}{\partial f}\right|_{(f_0,0)},
\end{equation}
is invertible. To simplify notations we denote the derivative with respect to $t$ of 
\eqref{phi-t-f} by
$$\dot{\phi}(f):=\left.\frac{\partial\phi}{\partial t}\right|_{(f,0)}.$$
By differentiating \eqref{sol-pi-2} with respect to $t$, we get
\begin{equation}\label{Lic-dot-phi}
(\mathcal{D}_{\phi_0}^{*}\uu(m_{\phi_0})\mathcal{D}_{\phi_0})(\dot{\phi}(f_0))+\uu(m_{\phi_0})\Lambda_{\phi_0}\omega+\big\langle(d\uu)(m_{\phi_0}),m_\omega\big\rangle=0.
\end{equation}
A straightforward calculation yields
\begin{align*}
\tilde{\mathcal{G}}_{\uu,\vv}(f,0)=&-\Pi_{\vv,\phi_0}\circ G_{\uu,\vv}(f),\\
G_{\uu,\vv}(f):=&\frac{\mathcal{D}_{\phi}^{*}\uu(m_{\phi})\mathcal{D}_{\phi}(\dot{\phi}(f))+\uu(m_{\phi})\Lambda_{\phi}\omega+\big\langle(d\uu)(m_{\phi}),m_\omega\big\rangle}{\vv(m_{\phi})},
\end{align*}
where $\phi:=\phi(f,0)$. For $\dot{f}\in\mathfrak{k}_{\phi_0}$ we denote $f_\varepsilon:=f_0+\varepsilon\dot{f}$ and $\phi_\varepsilon=f_\epsilon+\psi(f_\varepsilon,t)$, we then have
\begin{align}
\begin{split}\label{cxc}
&\langle Q_{\uu,\vv}(\dot{f}),\dot{f}\rangle_{\vv,\phi_0}:=\int_X\left(\left.\frac{d}{d\varepsilon}\right|_{\varepsilon=0}\tilde{\mathcal{G}}_{\uu,\vv}(f_\varepsilon,0)\right)\dot{f}\vv(m_{\phi_0})\omega_{\phi_0}^{[n]}\\
=&\left.\frac{d}{d\varepsilon}\right|_{\varepsilon=0}\int_X\Pi_{\vv,\phi_0}[G_{\uu,\vv}(f_{\varepsilon})]\dot{f}\vv(m_{\phi_0})\omega_{\phi_0}^{[n]}\\
=&-\left.\frac{d}{d\varepsilon}\right|_{\varepsilon=0}\int_X G_{\uu,\vv}(f_\varepsilon)\dot{f}\vv(m_{\phi_\varepsilon})\omega_{\phi_\varepsilon}^{[n]}-\int_X G_{\uu,\vv}(f_0)\dot{f}\left.\frac{d}{d\varepsilon}\right|_{\varepsilon=0}(\vv(m_{\phi_\varepsilon})\omega_{\phi_\varepsilon}^{[n]})\\
=&-\left.\frac{d}{d\varepsilon}\right|_{\varepsilon=0}\int_X G_{\uu,\vv}(f_\varepsilon)\dot{f}\vv(m_{\phi_\varepsilon})\omega_{\phi_\varepsilon}^{[n]} \quad\text{ (using \eqref{Lic-dot-phi} }G_{\uu,\vv}(f_0)=0)\\
=&-\left.\frac{d}{d\varepsilon}\right|_{\varepsilon=0}\int_X\big(\uu(m_{\phi_\varepsilon})\big(\mathcal{D}_{\phi_\varepsilon}(\dot{\phi}(f_\varepsilon)),\mathcal{D}_{\phi_\varepsilon}\dot{f}\big)_{\phi_\varepsilon}+[\uu(m_{\phi_\varepsilon})\Lambda_{\phi_\varepsilon}\omega+\big\langle(d\uu)(m_{\phi_\varepsilon}),m_\omega\big\rangle]\dot{f}\big)\omega_{\phi_\varepsilon}^{[n]}.
\end{split}
\end{align}
Using the following variational formulas,
\begin{align*}
&\left.\frac{d}{d\varepsilon}\right|_{\varepsilon=0}\omega_{\phi_\varepsilon}^{[n]}=-\Delta_{\phi_0}(\dot{f})\omega_{\phi_0}^{[n]}\\
&\left.\frac{d}{d\varepsilon}\right|_{\varepsilon=0}\uu(m_{\phi_\varepsilon})=\sum_{i=1}^{\ell}\uu_{,i}(m_{\phi_0})(d^c\dot{f})(\xi_i),\\
&\left.\frac{d}{d\varepsilon}\right|_{\varepsilon=0}\langle(d\uu)(m_{\phi_\varepsilon}),m_\omega\rangle=\sum_{i,j=1}^{\ell}\uu_{,ij}(m_{\phi_0})(d^c\dot{f})(\xi_j)m^{\xi_i}_\omega,
\\
&\left.\frac{d}{d\varepsilon}\right|_{\varepsilon=0}\Lambda_{\phi_\varepsilon}\omega=-(dd^{c}\dot{f},\omega)_{\phi_0},\\
&\left.\frac{d}{d\varepsilon}\right|_{\varepsilon=0}\mathcal{D}_{\phi_\varepsilon}\dot{f}=-\mathcal{D}_{\phi_0}|d\dot{f}|_{\phi_0}^{2}\quad\text{(see \Cref{Rem} below)},
\end{align*} 
and the following calculation from the proof of \cite[Lemma 4]{lahdili3} 
\begin{align*}
\left.\frac{d}{d\varepsilon}\right|_{\varepsilon=0}\int_X[\uu(m_{\phi_\varepsilon})\Lambda_{\phi_\varepsilon}\omega+\big\langle(d\uu)(m_{\phi_\varepsilon}),&m_\omega\big\rangle]\dot{f}\omega_{\phi_\varepsilon}^{[n]}
=\int_{X}\uu(m_{\phi_0})(\omega,d\dot{f}\wedge d^c\dot{f})_{\phi_0}\omega_{\phi_0}^{[n]}\\
&-\int_X\left(\Lambda_{\phi_0}\omega+\big\langle(d\uu)(m_{\phi_0}),m_\omega\big\rangle\right)|d\dot{f}|_{\phi_0}^{2}\omega_{\phi_0}^{[n]},
\end{align*}
we compute from \eqref{cxc}:
\begin{align*}
\langle Q_{\uu,\vv}(\dot{f}),\dot{f}\rangle_{\vv,\phi_0}=&\int_X(\mathcal{D}_{\phi_0}^{*}\uu(m_{\phi_0})\mathcal{D}_{\phi_0})(\dot{\phi}(f_0))|d\dot{f}|_{\phi_0}^{2}\omega_{\phi_0}^{[n]}\\
&-\int_{X}\uu(m_{\phi_0})(\omega,d\dot{f}\wedge d^c\dot{f})_{\phi_0}\omega_{\phi_0}^{[n]}\\
&+\int_X\left(\Lambda_{\phi_0}\omega+\big\langle(d\uu)(m_{\phi_0}),m_\omega\big\rangle\right)|d\dot{f}|_{\phi_0}^{2}\omega_{\phi_0}^{[n]}\\
=&-\int_{X}\uu(m_{\phi_0})(\omega,d\dot{f}\wedge d^c\dot{f})_{\phi_0}\omega_{\phi_0}^{[n]},
\end{align*} 
where we used \eqref{Lic-dot-phi} for the second equality. It follows that $Q_{\uu,\vv}$ is bejective on $\mathfrak{k}_{\phi_0}$. Therefore, by the implicit function theorem, there exist a path $(f(t))_{t\in(0,\epsilon)}\in\mathfrak{k}_{\phi_0}$, $f(0)=f_0$ such that $\mathcal{G}_{\uu,\vv}(f(t),t)=0$. From \eqref{sol-pi-2}, we obtain
 \begin{equation*}
 F_{\uu,\vv}(\phi(f(t),t),t)=0,
 \end{equation*}
 for any $t\in (0,\epsilon)$, which completes the proof.
\end{proof}
\begin{lem}\label{Rem}
We have
\begin{equation}\label{exp}
\left.\frac{d}{d\varepsilon}\right|_{\varepsilon=0}\mathcal{D}_{\phi_\varepsilon}\dot{f}=-\mathcal{D}_{\phi_0}|d\dot{f}|_{\phi_0}^{2}.
\end{equation}
\end{lem}
\begin{proof}
Using \cite[Lemma 1.23.2]{Gauduchon}, and the fact that $\dot{f}$ is a Killing potential we obtain,
\begin{align}
\begin{split}\label{ccc}
\left.\frac{d}{d\varepsilon}\right|_{\varepsilon=0}\mathcal{D}_{\phi_\varepsilon}\dot{f}=&-\frac{\sqrt{2}}{2}\omega_{\phi_0}((\mathcal{L}_{V}J)\cdot,\cdot)\text{ where }V:=\left.\frac{d}{d\varepsilon}\right|_{\varepsilon=0}\grad_{\phi_{\varepsilon}}(\dot{f})\\
=&-\frac{\sqrt{2}}{2}(\nabla^{\phi_0}V^{\flat})^{-},
\end{split}
\end{align}
where the musical isomorphisme used in $V^{\flat}$ is with respect to the metric $\omega_{\phi_0}$. Using the equation $\omega_{\phi_\varepsilon}(\grad_{\phi_{\varepsilon}}(\dot{f}),\cdot)=Jd\dot{f}$, we obtain
\begin{align*}
V=&(dd^{c}\dot{f})(J\grad_{\phi_{0}}(\dot{f}),\cdot)\\
=&\mathcal{L}_{J\grad_{\phi_{0}}(\dot{f})}d^c\dot{f}-d((d^{c}\dot{
f})(J\grad_{\phi_{0}}(\dot{f})))\text{ by Cartan formula,}\\
=&0-d|d\dot{f}|_{\phi_0}^{2} \text{ since $J\grad_{\phi_{0}}(\dot{f})$ is real holomorphic}.
\end{align*}
Substituting the above expression of $V$ back into \eqref{ccc}, the expression \eqref{exp} follows.
\end{proof}
Now we are in position to proof \Cref{Thm-uniq}
\begin{proof}[Proof of \Cref{Thm-uniq}]
Using \Cref{Thm-Def}, the proof of \Cref{Thm-uniq} is very similar to \cite[Corollary 1.3]{CPZ}. We give the argument for the sake of clarity. Suppose that $\varphi_0, \tilde{\varphi}_0\in \mathcal{K}(X,\omega)^{\T}$, such that $\omega_{\varphi_0},\omega_{\tilde{\varphi}_0}$ are two $\T$-invariant $(\uu,\vv)$-extremal metrics in the K\"ahler class $\alpha$. Using \Cref{Thm-Def}, we get two paths $\phi:[0,\epsilon)\times X\to \mathbb{R}$ and $\tilde{\phi}:[0,\epsilon)\times X\to \mathbb{R}$ in $\mathcal{K}(X,\omega)^{\T}$ such that $\phi_0$ (resp. $\tilde{\phi}_0$) is in the $G$-orbit of $\varphi_0$ (resp. $\tilde{\varphi}_0$) and $\phi_t$ (resp. $\tilde{\phi}_t$) solves \eqref{eq-uu-vv-ext-twist}. Notice that $\phi_t$ and $\tilde{\phi}_t$ are critical points of the functional $\mathcal{M}_{(\uu,\ell_{\rm ext}\cdot\vv)}^{t\omega}:=\mathcal{M}_{\uu,\vv}^{\rm rel}+t\mathcal{E}^{\omega}_{\uu}$.  Indeed,   by \eqref{E-theta} and \eqref{Mabuchi}, we have
\begin{align*}
&\big(d\mathcal{M}_{(\uu,\ell_{\rm ext}\cdot\vv)}^{t\omega}\big)_{\phi}(\dot{\phi})=\\
&-\int_X\Big(\frac{\Scal_{\uu}(\phi)-t\big(\uu(m_\phi)\Lambda_{\omega_{\phi}}\omega+\big\langle(d\uu)(m_{\phi}),m_\omega\big\rangle\big)}{\vv(m_{\phi})}-\ell_{\rm ext}(m_\phi)\Big)\dot{\phi}\vv(m_\phi)\omega^{[n]}_\phi.
\end{align*}
By convexity of $\mathcal{M}_{\uu,\vv}^{\rm rel}=\mathcal{M}_{(\uu,\ell_{\rm ext}\vv)}$ along weak geodesics \Cref{Mab-conv}, and strict convexity of $\mathcal{E}^{\omega}_{\uu}$ along weak geodesics \Cref{Cor-E-theta-strict}, it follows that the functional $\mathcal{M}_{(\uu,\ell_{\rm ext}\cdot\vv)}^{t\omega}$ is strictly convex on weak geodesics. Thus, $\phi=\tilde{\phi}$ on $(0,\epsilon)\times X$. As $\epsilon\to 0$ we obtain $\varphi_0=f^{*}\tilde{\varphi}_0$, for some $f\in G$.
\end{proof}

\begin{rem}
By \cite[Corollary B.1]{lahdili3}, a $(\uu,\vv)$-extremal metric $\omega$ is always invariant under the action of a maximal torus $\Autred$. We can thus take in \Cref{Thm-uniq} $\mathbb{T}$ to be a maximal torus. In this case $G={\rm Aut}_{\rm red}^{\T}(X)^{\circ}=\T^c$ the complixified torus. Indeed~\footnote{Thanks to V.~Apostolov for this argument.}, by \cite[Theorem B1]{lahdili3} the group $G$ is a reductive Lie group (at the level of Lie algebras we have ${\rm Lie}(G)=\mathfrak{k}\oplus J\mathfrak{k}$ where $\mathfrak{k}$ is the Lie algebra of the connected compact group $K:={\rm Isom}^{\T}_0(X,\omega)\cap G$). As $\T \subset K$  is simultaneously central and maximal, it follows that $\mathfrak{k}=\tor$, and thus $K=\T$ and $G=\T^{c}$. Thus, when $\T$ is maximal, any two $(\uu,\vv)$-extremal metrics $\omega_1,\omega_2\in\alpha$, there exist $f\in{\T}^c$ such that $\omega_2=f^{*}\omega_1$.
\end{rem}

\end{document}